\documentclass{amsart}                     

\usepackage{amsthm}
\usepackage{amsmath,amssymb,amsfonts,tikz}

\usepackage{enumitem}

\usepackage[a4paper,headheight=0.3cm]{geometry}
\usepackage{fancyhdr}
\usepackage{url}
\usepackage[pdftitle={Normal Form Backward Induction for Decision Trees with Coherent Lower Previsions},%
pdfauthor={Nathan Huntley and Matthias C. M. Troffaes},%
pdfkeywords={backward induction, decision tree, lower prevision, sequential decision making, choice function, maximality, E-admissibility, interval dominance, maximin, imprecise probability}]{hyperref}

\makeatletter

\def\normformnew#1[#2]{
\@ifnextchar({%
\decsuba[{#1}^{}_{#2}\left(]}%
{
{#1}^{}_{#2}}}
\def\decsuba[#1]#2{
\decsubb[#1]}

\def\decsubb[#1]#2,#3{
\@ifnextchar({
\decsuba[#1{}^{#2}_{#3}\left(]}%
{
\@ifnextchar){\decsubc[#1{}^{#2}_{#3}]}
{
\decsubd[#1]{#2}{#3}
}}}

\def\decsubc[#1]#2{
\@ifnextchar){
\decsubc[#1\right)]}%
{\@ifnextchar({
\decsuba[#1\right)\left(]}%
{
#1\right)}}}

\def\decsubd[#1]#2#3#4{
\@ifnextchar({
\decsuba[#1{}^{#2}_{#3#4}\left(]}%
{
\@ifnextchar){\decsubc[#1{}^{#2}_{#3#4}]}
{
\decsubd[#1]#2{#3#4}
}}}

\newcommand{\treenotation}[1]{\treenotationevenbetter{#1}}

\newcommand{\treenotationevenbetter}[1]{
\treenotationstart{#1}}

\def\treenotationstart#1{
\@ifnextchar[{
\treenotationsqsq{#1}}%
{
\@ifnextchar({
\treenotationsqro{#1}
}
{
#1}}}%

\def\hiddenindex{\phantom{1}}

\def\treenotationsqsq#1[#2]{%
\@ifnextchar[{
\treenotationsqsq{#1{}^{\hiddenindex}_{#2}}}%
{
\@ifnextchar({
\treenotationsqro{#1{}^{\hiddenindex}_{#2}}
}
{
#1{}^{\hiddenindex}_{#2}}}}%

\def\treenotationsqro#1(#2){%
\@ifnextchar[{
\treenotationrosq{#1{}^{#2}}}%
{
\@ifnextchar({
\treenotationroro{#1{}^{#2}_{\hiddenindex}}
}
{
#1{}_{\hiddenindex}^{#2}}}}%

\def\treenotationroro#1(#2){%
\@ifnextchar[{
\treenotationrosq{#1{}^{#2}}}%
{
\@ifnextchar({
\treenotationroro{#1{}^{#2}_{\hiddenindex}}
}
{
#1{}_{\hiddenindex}^{#2}}}}%

\def\treenotationrosq#1[#2]{%
\@ifnextchar[{
\treenotationsqsq{#1_{#2}}}%
{
\@ifnextchar({
\treenotationsqro{#1_{#2}}
}
{
#1_{#2}}}}%

\makeatother

\DeclareMathOperator{\opt}{opt}

\newcommand{\lpr}{\underline{P}}
\newcommand{\upr}{\overline{P}}

\newcommand{\pspace}{\Omega}

\newcommand{\domlinprevs}{\mathcal{M}}

\newcommand{\decnode}{\treenotation{N}}
\newcommand{\chancenode}{\treenotation{N}}

\newcommand{\event}{\treenotation{E}}

\newcommand{\treeevent}[1]{\mathrm{ev}({#1})}

\newcommand{\lprpref}[1]{>_{\lpr|{#1}}}
\newcommand{\notlprpref}[1]{{\not>}_{\lpr|{#1}}}
\newcommand{\intervalpref}[1]{\sqsupset_{\lpr | {#1}}}
\newcommand{\notintervalpref}[1]{{\not\sqsupset}_{\lpr | {#1}}}

\DeclareMathOperator{\back}{back}
\DeclareMathOperator{\optmax}{opt_{\lprpref{\cdot}}}
\DeclareMathOperator{\backmax}{back_{\optmax}}
\DeclareMathOperator{\backopt}{\back_{\opt}}
\DeclareMathOperator{\optE}{\opt_{\domlinprevs}}
\DeclareMathOperator{\backE}{\back_{\optE}}
\DeclareMathOperator{\optint}{\opt_{\sqsupset_{\lpr | \cdot}}}
\DeclareMathOperator{\optsucc}{\opt_{\succ|\cdot}}
\DeclareMathOperator{\optmaximin}{\opt_{\lpr}}

\newtheorem{lemma}{Lemma}
\newtheorem{proposition}[lemma]{Proposition}
\newtheorem{corollary}[lemma]{Corollary}
\newtheorem{theorem}[lemma]{Theorem}
\newtheorem{example}[lemma]{Example}
\newtheorem{definition}[lemma]{Definition}

\newtheorem{property}{Property}

\newcommand{\tree}{T}
\newcommand{\atree}{U}

\newcommand{\setoftrees}{\mathcal{T}}

\DeclareMathOperator{\subtreeatoper}{st}
\newcommand{\subtreeat}[2]{\subtreeatoper_{#2}(#1)}
\newcommand{\subtree}{\treenotation{\tree}}

\DeclareMathOperator{\normoper}{norm}
\DeclareMathOperator{\normgambles}{gamb}

\DeclareMathOperator{\nfd}{nfd}

\newcommand{\gambplus}{\oplus}
\newcommand{\bigdecnodeunion}{\bigsqcup}
\newcommand{\decnodeunion}{\sqcup}
\newcommand{\bigchancenodemixture}{\bigodot}
\newcommand{\chancenodemixture}{\odot}

\newcommand{\setplus}{\oplus}

\newcommand{\compl}[1]{\overline{#1}}

\begin{document}

\title[Normal Form Backward Induction for Decision Trees with Coh. Lower Previsions]{Normal Form Backward Induction for Decision Trees with Coherent Lower Previsions}
\author{Nathan Huntley}
\address{Durham University, Department of Mathematical Sciences, Science Laboratories, South Road, Durham DH1 3LE, United Kingdom}
\email{nathan.huntley@durham.ac.uk}
\author{Matthias C. M. Troffaes}
\address{Durham University, Department of Mathematical Sciences, Science Laboratories, South Road, Durham DH1 3LE, United Kingdom}
\email{matthias.troffaes@gmail.com}

\keywords{backward induction; decision tree; lower prevision; sequential decision making; choice function; maximality; E-admissibility; interval dominance; maximin; imprecise probability}

\begin{abstract}
We examine normal form solutions of decision trees under typical choice functions induced by lower previsions. For large trees, finding such solutions is hard as very many strategies must be considered. In an earlier paper, we extended backward induction to arbitrary choice functions, yielding far more efficient solutions, and we identified simple necessary and sufficient conditions for this to work. In this paper, we show that backward induction works for maximality and E-admissibility, but not for interval dominance and $\Gamma$-maximin. We also show that, in some situations, a computationally cheap approximation of a choice function can be used, even if the approximation violates the conditions for backward induction; for instance, interval dominance with backward induction will yield at least all maximal normal form solutions.
\end{abstract}

\maketitle
\thispagestyle{fancy}

\section{Introduction}
\label{sec:intro}

In classical decision theory, one aims to maximize expected utility. Such approach requires probabilities for all relevant events. However, when information and knowledge are limited, sadly, the decision maker may not be able to specify or elicit probabilities exactly. To handle this, various theories have been suggested, including lower previsions \cite{1991:walley}, which essentially amount to sets of probabilities.

In non-sequential problems, given a lower prevision, various generalizations of maximizing expected utility exist \cite{2007:troffaes}. Sequential extensions of some of these alternatives have been suggested~\cite{1988:seidenfeld,1999:harmanec,1999:jaffray,2001:augustin,2004:seidenfeld,2005:decooman,2005:kikuti,2007:seidenfeld}, yet not systematically studied. In this paper, we study, systematically, using lower previsions, which decision criteria admit efficient solutions to sequential decision problems, by backward induction, even if probabilities are not exactly known. Our main contribution is that we prove for which criteria backward induction coincides with the usual normal form.

We study very general sequential decision problems: a subject can choose from a set of options, where each option has uncertain consequences, leading to either rewards or more options. Based on her beliefs and preferences, the subject seeks an optimal strategy. Such problems are represented by a \emph{decision tree} 
\cite{1961:raiffa,1985:lindley,2001:clemen}.

When maximizing expected utility, one can solve a decision tree by the usual \emph{normal form method}, or by \emph{backward induction}. First, note that the subject can specify, in advance, her actions in all eventualities. In the normal form, she simply chooses a specification which maximizes her expected utility. However, in larger problems, the number of specifications is gargantuan, and the normal form is not feasible.

Fortunately, backward induction is far more efficient. We find the expected utility at the final decision nodes, and then replace these nodes with the maximum expected utility. The previously penultimate decision nodes are now ultimate, and the process repeats until the root is reached. Backward induction is guaranteed to coincide with the normal form \cite{1961:raiffa} if probabilities are non-zero~\cite[p.~44]{1988:hammond}.

The usual normal form method works easily with decision criteria for lower previsions: apply it to the set of all strategies. Generalizing backward induction is harder, as no single expectation summarizes all relevant information about substrategies, unlike with expected utility. We follow Kikuti et al.~\cite{2005:kikuti}, and instead replace nodes with sets of optimal substrategies, moving from right to left in the tree, eliminating strategies as we go. De~Cooman and Troffaes~\cite{2005:decooman} presented a similar idea for dynamic programming.

In this general setting, normal form and backward induction can differ, as noted by many \cite{1988:seidenfeld,1989:machina,2004:seidenfeld,2005:kikuti,2005:decooman,1999:jaffray,2001:augustin}. However, for some decision criteria the methods always coincide. In~\cite{2011:huntley:subtree:perfectness}, we found conditions for coincidence. In this paper, we expand the work begun in~\cite{2008:huntley:troffaes::impdectrees:smps}, and investigate what works for lower previsions, finding that \emph{maximality} and \emph{E-admissibility} work, but the others do not.

This coincidence is of interest for at least two reasons. First, as mentioned, the normal form is not feasible for larger trees, whereas backward induction can eliminate many strategies early on, hence being far more efficient. Secondly, one might argue that a solution where the two methods differ is philosophically flawed \cite{2011:huntley:subtree:perfectness,1988:hammond,1988:seidenfeld,1990:mcclennen}.

The paper is organized as follows. Section~\ref{sec:dectrees} explains decision trees and introduces notation. Section~\ref{sec:ip:optimality} presents lower previsions and their decision criteria, and demonstrates normal form backward induction on a simple example. Section~\ref{sec:solution} formally defines the two methods, and characterizes their equivalence, which is applied in Section~\ref{sec:ip:application} to lower previsions. Section~\ref{sec:example:two} discusses a larger example. Section~\ref{sec:conclusion} concludes. Readers familiar with decision trees and lower previsions can start with Sections~\ref{sec:example:one} and~\ref{sec:example:two}.

\section{Decision Trees}
\label{sec:dectrees}

\subsection{Definition and Example}

Informally, a decision tree \cite{1985:lindley,2001:clemen} is a
graphical causal representation of decisions, events, and rewards. Decision trees
consist of a rooted tree \cite[p.~92, Sec.~3.2]{1999:gross::graph:theory} of decision nodes, chance nodes, and reward leaves, growing from
left to right. The left hand side corresponds to what happens first,
and the right side to what happens last.

Consider the following example. Tomorrow, a subject is going for
a walk in the lake district. It may rain ($E_1$), or
not ($E_2$). The subject can either take a waterproof ($d_1$), or
not ($d_2$). But the subject may also choose to buy today's newspaper, at cost $c$, to learn about tomorrow's
weather forecast ($d_S$), or not ($d_{\compl{S}}$), before leaving
for the lake district. The
forecast has two possible outcomes: predicting rain ($S_1$), or not
($S_2$).

The corresponding decision tree is depicted in
Figure~\ref{fig:lake:tree:basic}.
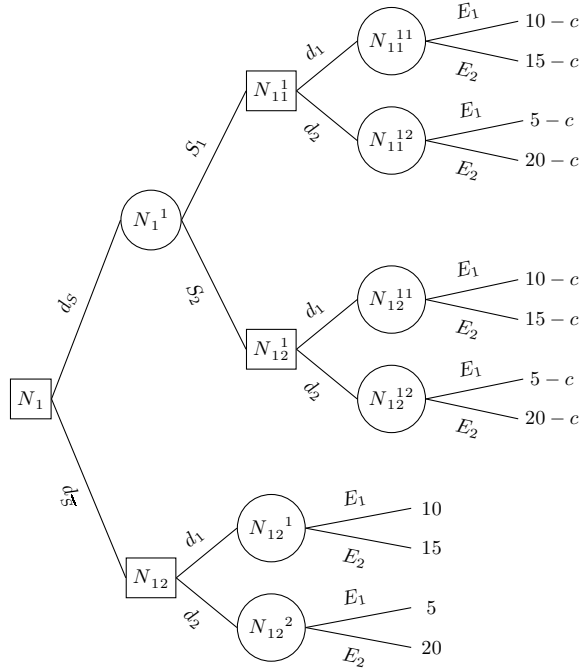
\begin{figure}
  \begin{center}
    \begin{tikzpicture}
      [minimum size=2em,parent anchor=east,child anchor=west,grow'=east,scale=0.75,transform shape]
        \node[draw,rectangle]{$\decnode[1]$}
        [sibling distance=18em,level distance=6em]
        child{
          node[draw,circle]{$\chancenode[1](1)$}
          [sibling distance=13em,level distance=6em]
          child{
            node[draw,rectangle]{$\decnode[1](1)[1]$}
            [sibling distance=5em,level distance=6em]
            child{
              node[draw,circle]{$\chancenode[1](1)[1](1)$}
              [sibling distance=2em,level distance=8em]
              child{
                node{$10-c$}
                edge from parent
                node[above,sloped]{$E_1$}
              }
              child{
                node{$15-c$}
                edge from parent
                node[below,sloped]{$E_2$}
              }
              node{}
              edge from parent
              node[above,sloped]{$d_1$}
            }
            child{
              node[draw,circle]{$\chancenode[1](1)[1](2)$}
              [sibling distance=2em,level distance=8em]
              child{
                node{$5-c$}
                edge from parent
                node[above,sloped]{$E_1$}
              }
              child{
                node{$20-c$}
                edge from parent
                node[below,sloped]{$E_2$}
              }
              edge from parent
              node[below,sloped]{$d_2$}
            }
            edge from parent
            node[above,sloped]{$S_1$}
          }
          child{
            node[draw,rectangle]{$\decnode[1](1)[2]$}
            [sibling distance=5em,level distance=6em]
            child{
              node[draw,circle]{$\chancenode[1](1)[2](1)$}
              [sibling distance=2em,level distance=8em]
              child{
                node{$10-c$}
                edge from parent
                node[above,sloped]{$E_1$}
              }
              child{
                node{$15-c$}
                edge from parent
                node[below,sloped]{$E_2$}
              }
              edge from parent
              node[above,sloped]{$d_1$}
            }
            child{
              node[draw,circle]{$\chancenode[1](1)[2](2)$}
              [sibling distance=2em,level distance=8em]
              child{
                node{$5-c$}
                edge from parent
                node[above,sloped]{$E_1$}
              }
              child{
                node{$20-c$}
                edge from parent
                node[below,sloped]{$E_2$}
              }
              edge from parent
              node[below,sloped]{$d_2$}
            }
            edge from parent
            node[below,sloped]{$S_2$}
          }
          edge from parent
          node[above,sloped]{$d_S$}
        }
        child{
          node[draw,rectangle]{$\decnode[1][2]$}
          [sibling distance=5em,level distance=6em]
          child{
            node[draw,circle]{$\chancenode[1][2](1)$}
            [sibling distance=2em,level distance=8em]
            child{
              node{$10$}
              edge from parent
              node[above,sloped]{$E_1$}
            }
            child{
              node{$15$}
              edge from parent
              node[below,sloped]{$E_2$}
            }
            edge from parent
            node[above,sloped]{$d_1$}
          }
          child{
            node[draw,circle]{$\chancenode[1][2](2)$}
            [sibling distance=2em,level distance=8em]
            child{
              node{$5$}
              edge from parent
              node[above,sloped]{$E_1$}
            }
            child{
              node{$20$}
              edge from parent
              node[below,sloped]{$E_2$}
            }
            edge from parent
            node[below,sloped]{$d_2$}
          }
          edge from parent
          node[below,sloped]{$d_{\compl{S}}$}
        };
    \end{tikzpicture}
    \caption{A decision tree for walking in the lake district.}
  \label{fig:lake:tree:basic}
  \end{center}
\end{figure}
Decision nodes are depicted by squares, and chance nodes by
circles. From each node, a number of branches emerge, representing decisions at decision nodes and events at chance nodes. The events from a node form a partition of the possibility space: exactly one of the events will take place. 
Each path in a decision tree corresponds to a sequence
of decisions and events. The reward from each such
sequence appears at the right hand end of the branch.

\subsection{Notation}
\label{sec:dectrees:notation}

A particular decision tree can be seen as a combination of smaller decision trees: for example, one could draw the subtree corresponding to buying the newspaper, and also draw the subtree corresponding to making an immediate decision. The decision tree for the full problem is then formed by joining these two subtrees at a decision node.

So, we can represent a decision tree by its subtrees and the type of its root node. Let $\tree_1$, \dots, $\tree_n$ be decision trees. If $\tree$ combines the trees at a decision node, we write
\begin{equation*}
  \tree = \bigdecnodeunion_{i=1}^n \tree_i.
\end{equation*}
If $\tree$ combines the trees at a chance node, with subtree $\tree_i$ being connected by event $\event[i]$ ($\event[1]$, \dots, $\event[n]$ is a partition of the possibility space) we write
\begin{equation*}
  \tree = \bigchancenodemixture_{i=1}^n \event[i] \tree_i.
\end{equation*}
For instance, for the tree of Fig.~\ref{fig:lake:tree:basic} with $c=1$, we write
\begin{equation*}
  (S_1(T_1\decnodeunion T_2)\chancenodemixture S_2(T_1\decnodeunion T_2))\decnodeunion(U_1\decnodeunion U_2)
\end{equation*}
with, where we denoted the reward nodes by their utility,
\begin{align*}
  T_1&=E_1 9 \chancenodemixture E_2 14
  &
  U_1&=E_1 10 \chancenodemixture E_2 15
  \\
  T_2&=E_1 4 \chancenodemixture E_2 19
  &
  U_2&=E_1 5 \chancenodemixture E_2 20
\end{align*}
The above notation shall prove very useful when considering recursive definitions.

In this paper we often consider subtrees of larger trees. For subtrees, it is important to know the events that were observed in the past. Two subtrees with the same configuration of nodes and arcs may have different preceding events, and should be treated differently. Therefore we associate with every decision tree $\tree$ an event $\treeevent{\tree}$ representing the intersection of all the events on chance arcs that have preceded $\tree$. 

\begin{definition}\label{def:subtreeat}
 A subtree of a tree $\tree$ obtained by removal of all non-descendants of a particular node $N$ is called \emph{the subtree of $\tree$ at $N$} and is denoted by $\subtreeat{\tree}{N}$.
\end{definition}

These subtrees are called `continuation trees' by Hammond \cite{1988:hammond}.

Consider all possible ways that sets of decision trees $\setoftrees_1,$ \dots, $\setoftrees_n$ can be combined. Our notation easily extends. For any partition $E_1$, \dots, $E_n$,
\begin{equation*}
\bigchancenodemixture_{i=1}^n E_i \setoftrees_i = \Bigg\{\bigchancenodemixture_{i=1}^n E_i \tree_i\colon \tree_i \in \setoftrees_i\Bigg\}.
\end{equation*}
For any sets of consistent decision trees $\setoftrees_1$, \dots, $\setoftrees_n$,
\begin{equation*}
\bigdecnodeunion_{i=1}^n \setoftrees_i = \Bigg\{\bigdecnodeunion_{i=1}^n \tree_i\colon\tree_i\in\setoftrees_i\Bigg\}.
\end{equation*}

For convenience we only work with decision trees for which there is no event arc that is impossible given preceding events.
\begin{definition}
  A decision tree $\tree$ is called \emph{consistent} if for every node $N$ of $\tree$,
  $$\treeevent{\subtreeat{\tree}{N}}\neq\emptyset.$$
\end{definition}
Clearly, if a decision tree $\tree$ is consistent, then for any node $N$ in $\tree$, $\subtreeat{\tree}{N}$ is also consistent. Considering only consistent trees is not really a restriction, since inconsistent trees would only be drawn due to an oversight and could easily be made consistent.

\subsection{Solving Decision Trees with Probabilities and Utilities}
\label{sec:preciseexample}

We give a brief overview of the standard method of solving a decision tree when probabilities of events are known. Suppose in Fig.~\ref{fig:lake:tree:basic} we have $p(S_1)=0.6$, $p(E_1|S_1)=0.7$, and $p(E_1|S_2)=0.2$, so $p(E_1)=0.5$. We first calculate the expected utility of the final chance nodes. For example, the expected utility at $\chancenode[1](1)[1](1)$ is $0.7(10-c)+0.3(15-c)=11.5-c$, and the expected utility at $\chancenode[1](1)[1](2)$ is $9.5-c$.

We now see that at $\decnode[1](1)[1]$ it is better to choose decision $d_1$. We then \emph{replace} $\decnode[1](1)[1]$ and its subtree with the expected utility of $\decnode[1](1)[1](1)$: $11.5-c$. Also follow this procedure for $\decnode[1](1)[2]$ and $\decnode[1][2]$, and the tree has been reduced by a stage. We find that $d_2$ is optimal at $\decnode[1](1)[2]$ with value $17-c$, and at $\decnode[1][2]$ both decisions are optimal with value $12.5$.

Next, take expected utility at $\chancenode[1](1)$, which is $0.6(11.5-c)+0.4(17-c)=13.7-c$. At $\decnode[1]$, we therefore take decision $d_S$ if $c\le 1.2$ and $d_{\compl{S}}$ if $c\ge 1.2$. This procedure is illustrated in Fig.~\ref{fig:lake:tree:eu:solution}, where the dashed lines indicate decision arcs that are rejected because their expected utility is too low (for any specific $c\neq1.2$, either $d_S$ or $d_{\compl{S}}$ would be dashed).

\begin{figure}
  \begin{center}
    \begin{tikzpicture}
      [minimum size=2em,parent anchor=east,child anchor=west,grow'=east,scale=0.75,transform shape]
        \node[draw,rectangle]{$\decnode[1]$}
        [sibling distance=18em,level distance=6em]
        child{
          node[draw,circle]{$13.7-c$}
          [sibling distance=13em,level distance=6em]
          child{
            node[minimum height=3.5em,draw,rectangle]{$11.5-c$}
            [sibling distance=5em,level distance=8em]
            child{
              node[draw,circle]{$11.5-c$}
              [sibling distance=2em,level distance=8em]
              child{
                node{$10-c$}
                edge from parent
                node[above,sloped]{$0.7$}
              }
              child{
                node{$15-c$}
                edge from parent
                node[below,sloped]{$0.3$}
              }
              node{}
              edge from parent
              node[above,sloped]{$d_1$}
            }
            child{
              node[draw,circle]{$9.5-c$}
              [sibling distance=2em,level distance=8em]
              child{
                node{$5-c$}
                edge from parent[solid]
                node[above,sloped]{$0.7$}
              }
              child{
                node{$20-c$}
                edge from parent[solid]
                node[below,sloped]{$0.3$}
              }
              edge from parent[dashed]
            }
            edge from parent
            node[above,sloped]{$0.6$}
          }
          child{
            node[minimum height=3em,draw,rectangle]{$17-c$}
            [sibling distance=5em,level distance=8em]
            child{
              node[draw,circle]{$14-c$}
              [sibling distance=2em,level distance=8em]
              child{
                node{$10-c$}
                edge from parent[solid]
                node[above,sloped]{$0.2$}
              }
              child{
                node{$15-c$}
                edge from parent[solid]
                node[below,sloped]{$0.8$}
              }
              edge from parent[dashed]
            }
            child{
              node[draw,circle]{$17-c$}
              [sibling distance=2em,level distance=8em]
              child{
                node{$5-c$}
                edge from parent
                node[above,sloped]{$0.2$}
              }
              child{
                node{$20-c$}
                edge from parent
                node[below,sloped]{$0.8$}
              }
              edge from parent
              node[below,sloped]{$d_2$}
            }
            edge from parent
            node[below,sloped]{$0.4$}
          }
          edge from parent
          node[above,sloped]{$d_S$}
        }
        child{
          node[draw,rectangle]{$12.5$}
          [sibling distance=5em,level distance=6em]
          child{
            node[draw,circle]{$12.5$}
            [sibling distance=2em,level distance=8em]
            child{
              node{$10$}
              edge from parent
              node[above,sloped]{$0.5$}
            }
            child{
              node{$15$}
              edge from parent
              node[below,sloped]{$0.5$}
            }
            edge from parent
            node[above,sloped]{$d_1$}
          }
          child{
            node[draw,circle]{$12.5$}
            [sibling distance=2em,level distance=8em]
            child{
              node{$5$}
              edge from parent
              node[above,sloped]{$0.5$}
            }
            child{
              node{$20$}
              edge from parent
              node[below,sloped]{$0.5$}
            }
            edge from parent
            node[below,sloped]{$d_2$}
          }
          edge from parent
          node[below,sloped]{$d_{\compl{S}}$}
        };
    \end{tikzpicture}
    \caption{Solving the Lake District problem with expected utility.}
  \label{fig:lake:tree:eu:solution}
  \end{center}
\end{figure}
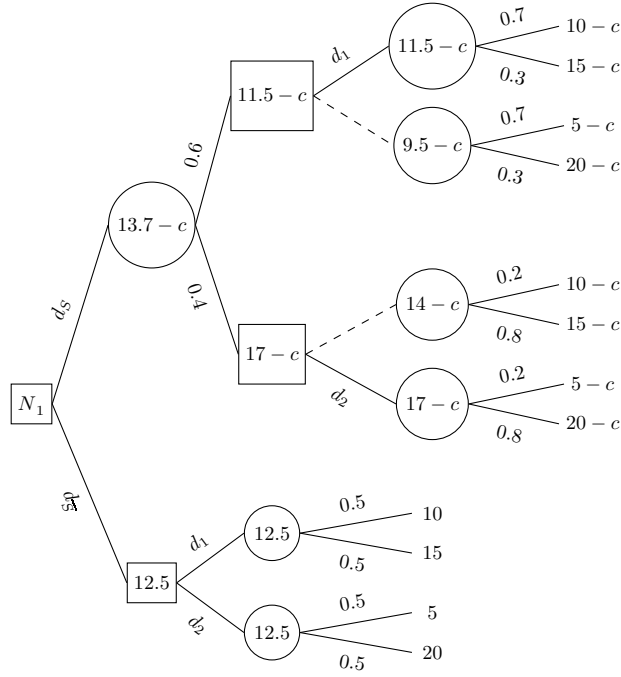

This method can only be carried out if the subject has assessed precise probabilities and utilities and wishes to maximize expected utility. It may be that the subject is unable or unwilling to comply with these requirements. The next section considers a possible solution, and demonstrates how backward induction can be generalized.

\section{Credal Sets and Lower Previsions}
\label{sec:ip:optimality}

First, we outline a straightforward generalization of the
theory of probability, allowing the subject to
model uncertainty in cases where too little information is available
to identify a unique probability distribution (see for instance \cite{1854:boole,1931:ramsey,1974:definetti,1961:smith,1975:williams:condprev,2007:williams:condprev,1991:walley}).

\subsection{Gambles, Credal Sets and Lower Previsions}

The \emph{possibility space} $\pspace$ is the set of all possible states of the world. Elements of
$\pspace$ are denoted by $\omega$.
Subsets of $\pspace$ are called \emph{events}, and are denoted by capital letters $A$, $B$, etc. The arcs emerging from chance nodes in a decision tree correspond to events.

A \emph{gamble} is a function $X\colon \pspace\to\mathbb{R}$, interpreted as an uncertain reward: should $\omega\in\pspace$ be the true state of the world, the gamble $X$ will yield the reward $X(\omega)$.

A \emph{probability mass function} is a non-negative real-valued
function $p\colon \pspace\to\mathbb{R}^+$ whose values sum to one \cite[p.~138,
Sec.~4.2]{2006:ross::probability}. For convenience, we will write
$p(A)$ for $\sum_{\omega\in A}p(\omega)$.

For the purpose of this paper, we assume that
\begin{itemize}
\item $\pspace$ is finite,
\item rewards are expressed in utiles,
\item the subject can express her beliefs by means of a closed convex
 set $\domlinprevs$ of probability mass functions $p$ ($\domlinprevs$ is called the \emph{credal set}), and
\item each probability mass function $p\in\domlinprevs$ satisfies $p(\omega)>0$, for all $\omega\in\pspace$.
\end{itemize}

Under the above assumptions, each $p$ in $\domlinprevs$ determines a conditional expectation
\begin{equation*}
  E_p(X|A)=\frac{\sum_{\omega\in A}X(\omega)p(\omega)}{p(A)},
\end{equation*}
and the whole set $\domlinprevs$ determines a conditional lower and upper expectation
\begin{align*}
  \lpr(X|A)&=\min_{p\in\domlinprevs}E_p(X|A)
  &
  \upr(X|A)&=\max_{p\in\domlinprevs}E_p(X|A),
\end{align*}
and this for every gamble $X$ and every non-empty event $A$.

The functional $\lpr$ is called a \emph{coherent conditional lower
  prevision}, and similarly, $\upr$ is called a \emph{coherent
  conditional upper prevision}. Although here we have defined these
by means of a set of probability measures, there are
different ways of obtaining and interpreting lower and upper
previsions (see for instance Miranda
\cite{2008:miranda::survey:lowprevs} for a survey).

Following De Finetti \cite{1974:definetti}, where it is convenient we denote the indicator
gamble
\begin{equation*}
  I_A(\omega)=
  \begin{cases}
    1&\text{if }\omega\in A\\
    0&\text{if }\omega\not\in A
  \end{cases}
\end{equation*}
of an event $A$ also simply by $A$: for instance, if $X$ is a gamble, and $A$ is an event, then $AX$ is just a shorthand notation for $I_{A}X$.

Below are some properties of coherent conditional lower and upper
previsions that we require later (see Williams \cite{1975:williams:condprev,2007:williams:condprev} or Walley \cite{1991:walley} for proofs).

\begin{proposition}\label{prop:lpr}
  For all non-empty events $A,B$, gambles $X,Y$, and constants $\lambda>0$:
  \begin{enumerate}[label=(\roman*)]
  \item\label{prop:lpr:ax-is-ay} If $AX=AY$ then $\lpr(X|A)=\lpr(Y|A)$.
  \item\label{prop:lpr:addition} $\lpr(X|A) + \lpr(Y|A) \leq \lpr(X+Y|A) \leq \lpr(X|A) + \upr(Y|A)$.
  \item\label{prop:lpr:multiplication} $\lpr(\lambda X|A) = \lambda \lpr(X|A) $ and $ \upr(\lambda X|A) = \lambda\upr(X|A)$.
  \item\label{prop:lpr:upr:conjugacy} $\lpr(X|A) = -\upr(-X|A)$.
  \item\label{prop:lpr:gbr} $\lpr(A(X-\lpr(X|A\cap B))|B)=0$.
  \end{enumerate}
\end{proposition}

Property~\ref{prop:lpr:gbr} generalizes the \emph{generalized Bayes
  rule}, which describes how conditional lower previsions are linked
to unconditional lower previsions. For instance, if $\lpr(A|B)>0$,
then $\lpr(\cdot|A\cap B)$ is uniquely
determined by $\lpr(\cdot|B)$ via property~\ref{prop:lpr:gbr}
\cite[p.~297, Thm.~6.4.1]{1991:walley}.

\subsection{Choice Functions and Optimality}
\label{sec:choice:functions}

Suppose a subject must choose a gamble from a set $\mathcal{X}$. In classical decision theory, a gamble is optimal if its expectation is maximal. More generally, given a credal set $\domlinprevs$, the subject might consider as optimal for example any gambles whose expectation is maximal for at least one $p\in\domlinprevs$: in fact, most criteria determine optimal decisions by comparison of gambles. So, we suppose that the subject has a way of determining an optimal subset of any set of gambles, conditional on any event $A$ (think of $A$ as $\treeevent{\tree}$):

\begin{definition}\label{def:choice:function}
  A \emph{choice function} $\opt$ is an operator that, for each non-empty event $A$, maps each non-empty finite set $\mathcal{X}$ of gambles to a non-empty subset of this set:
  \begin{equation*}
    \emptyset \neq \opt(\mathcal{X}|A) \subseteq \mathcal{X}.
  \end{equation*}
\end{definition}

Note that common uses of choice functions in social choice theory, such as by Sen~\cite[p.~63, ll.~19--21]{1977:sen}, do not consider conditioning on events, and define choice functions for arbitrary sets of options, rather than for gambles only.

The interpretation of a choice function is that when the subject can choose among the elements of $\mathcal{X}$, having observed $A$, she would only choose from $\opt(\mathcal{X}|A)$. Therefore, we say the elements of $\opt(\mathcal{X}|A)$ are \emph{optimal} (relative to $\mathcal{X}$ and $A$). Note that the subject may not consider them equivalent: adding a small incentive to choose a particular optimal option would not necessarily make it the single preferred option.

We now consider four popular choice functions that have been proposed for choosing between gambles given a coherent lower prevision. Further discussion of the criteria presented here can be found in Troffaes \cite{2007:troffaes}.

\subsubsection{Maximality}

Maximality is based on the following strict partial preference order $\lprpref{A}$.
\begin{definition}\label{def:lpr:partial:order}
  Given a coherent lower prevision $\lpr$, for any two gambles $X$ and $Y$ we write $X \lprpref{A} Y$ whenever $\lpr(X - Y|A)>0$.
\end{definition}

The partial order $\lprpref{A}$ gives rise to the choice function \emph{maximality}, proposed by Condorcet \cite[pp.~lvj--lxix, 4.$^{e}$ Exemple]{1785:condorcet}, Sen \cite{1977:sen}, and Walley \cite{1991:walley}, among others.\footnote{Because all probabilities in $\mathcal{M}$ are assumed to be strictly positive, Walley's admissibility condition is implied and hence omitted in Definition~\ref{def:maximality}.}
\begin{definition}
  \label{def:maximality}
  For any non-empty finite set of gambles $\mathcal{X}$ and each event $A\neq\emptyset$,
  \begin{equation*}
    \optmax(\mathcal{X}|A)=\{X\in\mathcal{X}\colon(\forall Y \in \mathcal{X})(Y\notlprpref{A} X)\}.
  \end{equation*}
\end{definition}

\subsubsection{E-admissibility}

Another criterion is E-admissibility, proposed by Levi \cite{1980:levi}. Recall that $\lpr(\cdot | A)$ is the lower envelope of $\domlinprevs$. For each $p\in\domlinprevs$ we can maximize expected utility:
\begin{equation*}
  \opt_p(\mathcal{X}|A)=\{X\in\mathcal{X}\colon(\forall Y\in\mathcal{X})(E_p(Y|A)\leq E_p(X|A))\}.
\end{equation*}
Then the set of E-admissible options is defined by:
\begin{definition}
  For any non-empty finite set of gambles $\mathcal{X}$ and each event $A\neq\emptyset$,
  \begin{equation*}
    \optE(\mathcal{X}|A) = \bigcup_{p\in\domlinprevs}\opt_p(\mathcal{X}|A).
  \end{equation*}
\end{definition}

A gamble $X$ is therefore E-admissible when it maximizes expected utility under at least one $p\in\domlinprevs$. Any E-admissible gamble is maximal~\cite[p.~162, ll.~26--28]{1991:walley}.

\subsubsection{Interval Dominance}

Interval dominance is based on the strict partial preference order $\intervalpref{A}$.
\begin{definition}
Given a coherent lower prevision $\lpr$, for any non-empty event $A$ and any two $A$-consistent gambles $X$ and $Y$ we write $X \intervalpref{A} Y$ whenever $\lpr(X|A) > \upr(Y|A)$.
\end{definition}
This ordering induces a choice function usually called \emph{interval dominance} \cite{2003:zaffalon::dementia,2007:troffaes}:
\begin{definition}
  For any non-empty finite set of gambles $\mathcal{X}$ and each event $A\neq\emptyset$,
\begin{equation*}
 \optint(\mathcal{X}|A) = \{X\in\mathcal{X}\colon(\forall Y \in \mathcal{X}) (Y \notintervalpref{A} X)\}.
\end{equation*}
\end{definition}

The above criterion was apparently first introduced by Kyburg \cite{1983:kyburg} and was originally called stochastic dominance. 
\subsubsection{$\Gamma$-maximin}

$\Gamma$-maximin selects gambles that maximizes the minimum expected reward.
\begin{definition}
  For any non-empty finite set of gambles $\mathcal{X}$ and each event $A\neq\emptyset$,
  \begin{equation*}
  \optmaximin(\mathcal{X}|A)=\{X\in\mathcal{X}\colon(\forall Y\in\mathcal{X})(\lpr(X|A)\geq\lpr(Y|A))\}.
  \end{equation*}
\end{definition}

$\Gamma$-maximin is a total preorder, and so usually selects a single gamble regardless of the degree of uncertainty in $\lpr$. $\Gamma$-maximin can be criticized for being too conservative (see Walley \cite[p.~164]{1991:walley}), as it only takes into account the worst possible scenario.

\subsection{Sequential Problems Using Lower Previsions}
\label{sec:example:one}

Consider again the lake district problem depicted in Fig.~\ref{fig:lake:tree:basic}, but now suppose that the subject has specified a coherent lower prevision, instead of a singe probability measure. For this example, we consider an $\epsilon$-contamination model: with probability $1-\epsilon$, observations follow a given probability mass function $p$, and with probability $\epsilon$, observations follow an unknown arbitrary distribution. One can easily check that, under this model, the lower expectation for a gamble $X$ is
\begin{equation*}
 \lpr(X) = (1-\epsilon) E_p(X) + \epsilon \inf X
\end{equation*}
The conditional lower expectation is \cite[p. 309]{1991:walley}
\begin{equation*}
 \lpr(X | A) = \frac{(1-\epsilon)E_p(AX) + \epsilon\inf_{\omega \in A} X(\omega)}{(1-\epsilon)p(A) + \epsilon}
\end{equation*}
As before, let $p(S_1)=0.6$, $p(E_1|S_1)=0.7$, and $p(E_1|S_2)=0.2$, so $p(E_1)=0.5$. Let $\epsilon = 0.1$.

Naively, she could solve the problem with the usual \emph{normal form method}: she lists all possible strategies (actions to take in all eventualities), finds the corresponding gambles, and applies a suitable choice function, say maximality.

Table~\ref{tab:lake:gambles} lists all strategies and their gambles. Each strategy gives a reward determined entirely by $\omega$, and hence has a corresponding gamble. For example, the gamble for the last strategy is
$$(5-c)S_1E_1+(20-c)S_1E_2 + (10-c)S_2E_1+(15-c)S_2E_2=S_1Y+S_2X-c,$$
with $X=10E_1+15E_2$ and $Y=5E_1+20E_2$.
Recall that $(5-c)S_1E_1$ is just a shorthand notation for $(5-c)I_{S_1}I_{E_1}$, and similarly for all other terms.

\begin{table}
  \begin{center}
    \begin{tabular}{l|l}
      strategy & gamble \\
      \hline
      $d_{\compl{S}}$, then $d_1$ & $X$ \\
      $d_{\compl{S}}$, then $d_2$ & $Y$ \\
      $d_S$, then $d_1$ if $S_1$ and $d_1$ if $S_2$ & $X-c$ \\
      $d_S$, then $d_2$ if $S_1$ and $d_2$ if $S_2$ & $Y-c$ \\
      $d_S$, then $d_1$ if $S_1$ and $d_2$ if $S_2$ & $S_1X+S_2Y-c$ \\
      $d_S$, then $d_2$ if $S_1$ and $d_1$ if $S_2$ & $S_1Y+S_2X-c$
    \end{tabular}
    \caption{Strategies and gambles for the lake district problem.}
   \label{tab:lake:gambles}
  \end{center}
\end{table}

Maximality can then be applied to find the optimal gambles: this requires comparison of all six gambles at once. Skipping the details of this calculation, for instance with $c=0.5$, we find that we should buy the newspaper and follow its advice.

However, could we think of a backward induction scheme which might not require comparison of all six gambles at once? Obviously, any such scheme will not work as easily as in Section~\ref{sec:preciseexample}, because we are not maximizing for a single number (i.e. expected utility). Instead, \emph{we retain the optimal strategies in subtrees}, as illustrated in Fig.~\ref{fig:lake:tree:solution}.

\begin{figure}
  \begin{center}
    \begin{minipage}{0.5\textwidth}
    \begin{enumerate}[label=(\roman*)]
    \item
    \begin{tikzpicture}
      [baseline=0,minimum size=2em,parent anchor=east,child anchor=west,grow'=east,scale=0.75,transform shape]
        \node[draw,rectangle]{$\decnode[1]$}
        [sibling distance=6em,level distance=6em]
        child{
          node[draw,circle]{$\chancenode[1](1)$}
          [sibling distance=5em,level distance=6em]
          child{
            node[draw,rectangle]{$\decnode[1](1)[1]$}
            [sibling distance=2em,level distance=6em]
            child{
              node{$\{X-c\}$}
              edge from parent
              node[above,sloped]{$d_1$}
            }
            child{
              node{$\{Y-c\}$}
              edge from parent
              node[below,sloped]{$d_2$}
            }
            edge from parent
            node[above,sloped]{$S_1$}
          }
          child{
            node[draw,rectangle]{$\decnode[1](1)[2]$}
            [sibling distance=2em,level distance=6em]
            child{
              node{$\{X-c\}$}
              edge from parent
              node[above,sloped]{$d_1$}
            }
            child{
              node{$\{Y-c\}$}
              edge from parent
              node[below,sloped]{$d_2$}
            }
            edge from parent
            node[below,sloped]{$S_2$}
          }
          edge from parent
          node[above,sloped]{$d_S$}
        }
        child{
          node[draw,rectangle]{$\decnode[1][2]$}
          [sibling distance=2em,level distance=6em]
          child{
            node{$\{X\}$}
            edge from parent
            node[above,sloped]{$d_1$}
          }
          child{
            node{$\{Y\}$}
            edge from parent
            node[below,sloped]{$d_2$}
          }
          edge from parent
          node[below,sloped]{$d_{\compl{S}}$}
        };
    \end{tikzpicture}
    \item
    \begin{tikzpicture}
      [baseline=0,minimum size=2em,parent anchor=east,child anchor=west,grow'=east,scale=0.75,transform shape]
        \node[draw,rectangle]{$\decnode[1]$}
        [sibling distance=3em,level distance=6em]
        child{
          node[draw,circle]{$\chancenode[1](1)$}
          [sibling distance=2em,level distance=6em]
          child{
            node[right]{$\opt(\{X-c,Y-c\}|S_1)=\{X-c\}$}
            edge from parent
            node[above,sloped]{$S_1$}
          }
          child{
            node[right]{$\opt(\{X-c,Y-c\}|S_2)=\{Y-c\}$}
            edge from parent
            node[below,sloped]{$S_2$}
          }
          edge from parent
          node[above,sloped]{$d_S$}
        }
        child{
          node[right]{$\opt(\{X,Y\})=\{X,Y\}$}
          edge from parent
          node[below,sloped]{$d_{\compl{S}}$}
        };
    \end{tikzpicture}
    \item
    \begin{tikzpicture}
      [baseline=0,minimum size=2em,parent anchor=east,child anchor=west,grow'=east,scale=0.75,transform shape]
        \node[draw,rectangle]{$\decnode[1]$}
        [sibling distance=2em,level distance=6em]
        child{
          node[right]{$\opt(\{S_1(X-c)+S_2(Y-c)\})=\{S_1X+S_2Y-c\}$}
          edge from parent
          node[above,sloped]{$d_S$}
        }
        child{
          node[right]{$\{X,Y\}$}
          edge from parent
          node[below,sloped]{$d_{\compl{S}}$}
        };
    \end{tikzpicture}
    \item \label{fig:lake:tree:solution:last:step}
    \begin{tikzpicture}
      [baseline=0,minimum size=2em,parent anchor=east,child anchor=west,grow'=east,scale=0.75,transform shape]
      \node{$\opt(\{S_1X+S_2Y-c,X,Y\})=\begin{cases}\{S_1X+S_2Y-c\}&\text{if }c<29/50\\ \{X,Y\} &\text{if }c>79/50 \\ \{S_1X+S_2Y-c,X,Y\} & \text{otherwise}\end{cases}$};
    \end{tikzpicture}
    \end{enumerate}
    \end{minipage}
    \caption{Solving the lake district example by normal form backward induction.}
  \label{fig:lake:tree:solution}
  \end{center}
\end{figure}
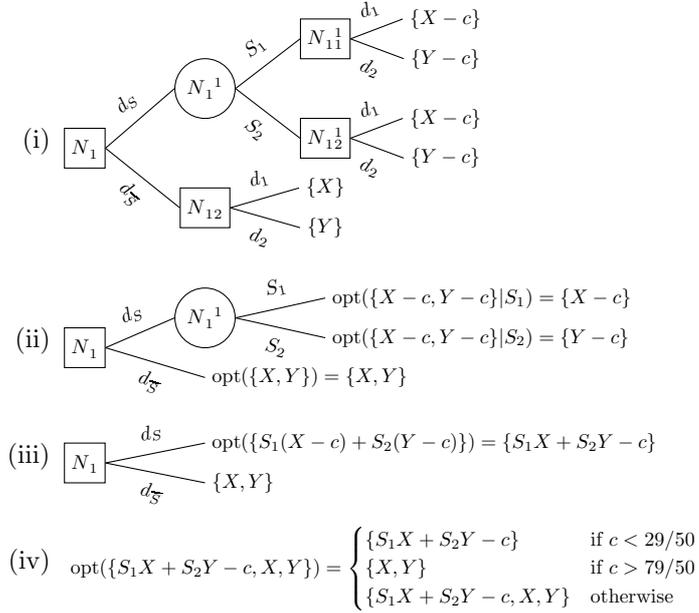

Denote subtrees at a particular node $N_*^*$ by
  $\subtree(*)[*]=\subtreeat{\tree}{N_{*}^*}$.

\begin{enumerate}[label=(\roman*)]
\item First, write down the gambles at the final chance nodes. For example, at $\chancenode[1](1)[1](1)$ the gamble is $(10-c)E_1 + (15-c)E_2=X-c$, and similarly for all others.

\item Let us first deal with the branch corresponding to refusing the newspaper. At the decision node $\decnode[1][2]$, we have a choice between two strategies that correspond to the gambles $X$ and $Y$. We also have $\treeevent{\subtree[1][2]}=\pspace$. So to determine the optimal strategies in this subtree, we must compare these two gambles unconditionally:
\begin{equation*}
 \lpr(X-Y) =  \lpr(Y-X) = -5 \epsilon=-1/2,
\end{equation*}
so at $\subtree[1][2]$ the strategies $d_1$ and $d_2$ are both optimal.

Now we move to the branch corresponding to buying the newspaper. At $\decnode[1](1)[1]$, we need to compare $X-c$ and $Y-c$. We have that $\treeevent{\subtree[1](1)[1]} = S_1$, and
\begin{equation*}
 \lpr((X-c) - (Y-c)| S_1) =
 \tfrac{6 - 31 \epsilon}{3+2\epsilon}>0,
\end{equation*}
so $X-c\lprpref{S_1}Y-c$ and the uniquely optimal strategy is $d_1$. Next, considering $\decnode[1](1)[2]$, we see that $\treeevent{\subtree[1](1)[2]} = S_2$, and
\begin{equation*}
   \lpr((Y-c) - (X-c) | S_2) =
 \tfrac{6 - 31 \epsilon}{2+3\epsilon}>0,
\end{equation*}
so the optimal strategy here is $d_2$.

\item Moving to $\chancenode[1](1)$, we see that only one of the original four strategies remains: ``$d_1$ if $S_1$ and $d_2$ if $S_2$'', corresponding to the gamble $S_1X+S_2Y-c$.
\item Finally, considering the entire tree $\tree$, three strategies are left: ``$d_1$ if $S_1$ and $d_2$ if $S_2$''; ``$d_{\compl{S}}$, then $d_1$''; ``$d_{\compl{S}}$, then $d_2$''. Therefore we need to find 
\begin{equation*}
\optmax(\{S_1X+S_2Y-c,X,Y\}).
\end{equation*}
 We have
 \begin{align*}
 \lpr\left(X - (S_1X+S_2Y-c)\right) &= c-(6+19\epsilon)/5=c-79/50 \\
\lpr\left((S_1X+S_2Y-c) - X\right) &= (6-31\epsilon)/5-c=29/50-c \\
 \lpr\left(Y - (S_1X+S_2Y-c)\right) &= c-(6+19\epsilon)/5=c-79/50 \\
\lpr\left((S_1X+S_2Y-c) - Y\right) &= (6-31\epsilon)/5-c=29/50-c
\end{align*}

Concluding (see Fig.~\ref{fig:lake:tree:solution}\ref{fig:lake:tree:solution:last:step}): 
\begin{itemize}
\item if the newspaper costs less than $29/50$, we should buy and follow its advice. 
\item if it costs more than $79/50$, we do not buy, but have insufficient information to decide whether to take the waterproof or not. 
\item if the newspaper costs between $29/50$ and $79/50$, we can take any of the three remaining options.
\end{itemize}
\end{enumerate}

Comparing this with the solution calculated in Section~\ref{sec:preciseexample}, we observe that the imprecision has created a range of $c$ for which it is unclear whether buying the newspaper is better than not, rather than the single value for $c$ in the precise case. Despite this, should the subject decide to buy the newspaper, she will follow the same policy in both cases: take the waterproof only if the newspaper predicts rain. Finally it should be noted that, although in both cases both $d_{\compl{S}}d_1$ and $d_{\compl{S}}d_2$ are involved in optimal normal form decisions for some values of $c$, in the precise case this is because they are \emph{equivalent} and in the imprecise case they are \emph{incomparable}. A tiny increase in value for, say, not taking the waterproof and no rain, would make $d_{\compl{S}}d_1$ always non-optimal under $E_p$ but still optimal under $\lpr$ for $c\geq29/50$.

For this particular example, it is easy though tedious to check that backward induction gives the same answer as the usual normal form method (that is, applying maximality to all six gambles), for any value of $c$. It is easy to find choice functions and decision trees where this does not work~\cite{1986:lavalle,1999:jaffray,2004:seidenfeld}.
We want to know for which of the coherent lower previsions choice functions the two methods agree. To answer this, we invoke a theorem relating to general choice functions, outlined in the next section. 

\section{Normal Form Solutions for Decision Trees}
\label{sec:solution}

We now introduce the necessary terminology for examining our two methods of solution in detail, and provide theorems stating when they coincide. These two methods yield \emph{normal form solutions}, that is, sets of optimal strategies at the root node. 

\subsection{Normal Form Decisions, Solutions, Operators, and Gambles}
\label{sec:forms}

Suppose the subject specifies a decision for each eventuality, and resolves to follow this policy. She now has no choice at decision nodes, and her reward is entirely determined by the state of nature. This corresponds to a reduced decision tree obtained by taking the initial tree and removing all but one of the arcs at each decision node. Such a reduced tree is called a \emph{normal form decision}, and represents what we called a ``strategy'' in Section~\ref{sec:example:one}. We denote the set of all normal form decisions of $\tree$ by $\nfd(\tree)$.

It is unlikely that the subject can specify a single optimal normal form decision for all problems. Nevertheless, she might be able to eliminate some unacceptable ones: a \emph{normal form solution} of a decision tree $\tree$ is simply a non-empty subset of $\nfd(\tree)$. 
A \emph{normal form operator} is then a function mapping every decision tree to a normal form solution of that tree. The two methods we investigate are normal form operators.

As we saw in Section~\ref{sec:example:one}, the reward for a normal form decision is determined entirely by the events that take place. That is, a normal form decision has a corresponding gamble, which we call a \emph{normal form gamble}. The set of all normal form gambles associated with a decision tree $\tree$ is denoted by $\normgambles(\tree)$, so $\normgambles$ is an operator on trees that yields the set of all gambles induced by normal form decisions of the tree.

We will need to know when a set of gambles can be represented by a consistent decision tree (as defined earlier in Section~\ref{sec:dectrees:notation}): 
\begin{definition}\label{def:gambles:consistent}
  Let $A$ be any non-empty event, and let $\mathcal{X}$ be a set of gambles. Then the following conditions are equivalent; if they are satisfied, we say that $\mathcal{X}$ is \emph{$A$-consistent}.
  \begin{enumerate}[label=(\Alph*)]
  \item\label{def:gambles:consistent:via:trees} There is a consistent decision tree $T$ with $\treeevent{\tree}=A$ and $\normgambles(\tree)=\mathcal{X}$.
  \item\label{def:gambles:consistent:via:inverse:map} For all $r\in\mathbb{R}$ and all $X\in\mathcal{X}$ such that $X^{-1}(r)\neq\emptyset$, it holds that $X^{-1}(r)\cap A\neq\emptyset$.
  \end{enumerate}
\end{definition}
Proof of equivalence of \ref{def:gambles:consistent:via:trees} and
\ref{def:gambles:consistent:via:inverse:map} is fairly
straightforward, whence omitted here.

The following notation proves convenient for normal form gambles at chance nodes.
\begin{definition}\label{def:gambsum:for:sets}
 For any events $\event_1$, \dots, $\event_n$ which form a partition, and any finite family of sets of gambles $\mathcal{X}_1$, \dots, $\mathcal{X}_n$, we define the following set of gambles:
\begin{equation}
 \label{eq:def:setsum}
 \sum_{i=1}^n \event_i\mathcal{X}_i
 =
 \left\{\sum_{i=1}^n \event_i X_i\colon X_i\in\mathcal{X}_i\right\}
\end{equation}
\end{definition}

\subsection{Normal Form Operator Induced by a Choice Function}

We can now formalize the simple normal form method described at the start of Section~\ref{sec:example:one}. Listing all strategies corresponds to finding $\nfd(\tree)$. Listing their corresponding gambles corresponds to finding $\normgambles(\tree)$. Then calculate $\opt(\normgambles(\tree)|\treeevent{\tree})$, and find all elements of $\nfd(\tree)$ that induced these optimal gambles. The solution is then the set of all these normal form decisions. Formally,

\begin{definition}\label{def:normoper}
  Given any choice function $\opt$, and any decision tree $\tree$ with $\treeevent{\tree}\neq\emptyset$,
\begin{equation*}
  \normoper_{\opt}(\tree)=\{\atree \in \nfd(\tree)\colon \normgambles(\atree)\subseteq \opt(\normgambles(\tree)|\treeevent{\tree})\}.
\end{equation*}
\end{definition}
The following important equality follows immediately:
\begin{equation}\label{eq:gambofnormoptisopt}
  \normgambles(\normoper_{\opt}(\tree))=\opt(\normgambles(\tree)|\treeevent{\tree}).
\end{equation}

Let us demonstrate this definition on lake district problem. For any
particular strategy $U$, say for instance ``buy the newspaper and take
the waterproof only if the newspaper predicts rain'', we can calculate its associated gamble $U$, which is in our instance
$$
X=(10-c)E_1S_1 \gambplus (15-c)E_2S_1\gambplus (5-c)E_1S_2\gambplus(20-c)E_2S_2.$$
We check
whether $X$ is optimal in the set of all gambles associated with
$\tree$, that is, whether
$X\in\opt(\normgambles(\tree)|\treeevent{\tree})$. If so, then
$\normgambles(U)=\{X\}\subseteq\opt(\normgambles(\tree)|\treeevent{\tree})$
and so $U\in\normoper_{\opt}(\tree)$. Otherwise,
$\normgambles(U)=\{X\}\not\subseteq\opt(\normgambles(\tree)|\treeevent{\tree})$
and so $U\not\in\normoper_{\opt}(\tree)$. This procedure for each
strategy in $\tree$ will determine $\normoper_{\opt}(\tree)$.

\subsection{Normal Form Backward Induction}

The operator $\normoper_{\opt}$ is a natural and popular choice, but for practical or philosophical reasons one may wish to be able to find it by backward induction. Even for an almost trivial problem such as Fig.~\ref{fig:lake:tree:basic} there are already six normal form decisions. If a tree $\tree$ has at least $n$ decision nodes in every path from the root to any leaf, and each decision node has at least two children, then there will be at least $2^n$ normal form decisions associated with $\tree$ (and often a lot more). Working with sets of $2^n$ gambles may be impractical for large $n$, particularly for maximality and E-admissibility, so a method that may avoid applying the choice function on a large set is necessary.

Implementation of backward induction is easy when there is a unique choice at every node, but a choice function may not have this property, so we need to adapt the traditional approach. The technique informally introduced in Section~\ref{sec:example:one} is a generalization of the method of Kikuti et al.~\cite{2005:kikuti}, where the only difference is that we apply our choice function at all nodes, not just decision nodes. Although the focus of Kikuti et al. is also on uncertainty models represented by coherent lower previsions, their approach, and so our generalization, can be used for any choice function on gambles.

The goal of our backward induction algorithm is to reach a normal form solution of $\tree$ by finding normal form solutions of subtrees of $\tree$, and using these to remove some elements of $\nfd(\tree)$ before applying $\opt$. A formal definition requires many definitions that hinder clarity, so in this paper we prefer a more intuitive informal approach. Rigorous treatment of our backward induction operator can be found in \cite{2009:huntley:troffaes::isipta,2011:huntley:subtree:perfectness}.

The algorithm moves from right to left in the tree as follows. At a subtree $\subtreeat{\tree}{N}$, find the set of normal form decisions $\nfd(\tree)$, but remove any strategies that contain substrategies judged non-optimal at any descendent node. For example, in Fig.~\ref{fig:lake:tree:basic}, at $\chancenode[1](1)$ there are four normal form decisions, but in the example the substrategy $d_2$ was removed at $\decnode[1](1)[1]$, and $d_1$ was removed at $\decnode[1](1)[2]$, so the only strategy at $\chancenode[1](1)$ that is retained is $d_1$ if $S_1$, $d_2$ if $S_2$. Next, find the corresponding gambles of all surviving normal form decisions, apply $\opt$, and transform back to optimal normal form decisions. Move to the next layer of nodes, and continue until the root node is reached. This yields a set of normal form decisions at the root node, that is, the algorithm corresponds to a normal form operator, which we call $\backopt$.

A further example using maximality, but this time using the notation of this section, can be found in Section~\ref{sec:example:two}, if further clarification is required.

In~\cite{2009:huntley:troffaes::isipta,2011:huntley:subtree:perfectness}, we found four necessary and sufficient properties on $\opt$ for $\backopt$ and $\normoper_{\opt}$ to coincide for any consistent decision tree.

\begin{property}[Backward conditioning property]\label{prop:backward:conditioning:property}
  Let $A$ and $B$ be events such that $A\cap B\neq\emptyset$ and $\compl{A}\cap B \neq \emptyset$, and let $\mathcal{X}$ be a non-empty finite $A\cap B$-consistent set of gambles, with $\{X,Y\}\subseteq \mathcal{X}$ such that $AX=AY$. Then $X\in\opt(\mathcal{X}|A\cap B)$ implies $Y\in\opt(\mathcal{X}|A\cap B)$ whenever there is a non-empty finite $\compl{A}\cap B$-consistent set of gambles $\mathcal{Z}$ such that, for at least one $Z\in\mathcal{Z}$,
  \begin{equation*}
    AX + \compl{A}Z \in\opt(A\mathcal{X} + \compl{A}\mathcal{Z} | B).
  \end{equation*}
\end{property}

This property requires that, if two gambles agree on $A$, it is not possible for exactly one to be optimal, conditional on any subset of $A$, unless there is no suitable $\mathcal{Z}$.

\begin{property}[Insensitivity of optimality to the omission of non-optimal elements]\label{prop:insensitivity:to:non:optimal:elements}
For any event $A\neq\emptyset$, and any non-empty finite $A$-consistent sets of gambles $\mathcal{X}$ and $\mathcal{Y}$,
\begin{equation*}
\opt(\mathcal{X}|A) \subseteq \mathcal{Y} \subseteq \mathcal{X} \Rightarrow \opt(\mathcal{Y}|A) = \opt(\mathcal{X}|A).
\end{equation*}
\end{property}
If $\opt$ satisfies this property, then removing non-optimal elements from a set does not affect whether or not each of the remaining elements is optimal. The property is called `insensitivity to the omission of non-optimal elements' by De~Cooman and Troffaes \cite{2005:decooman}, and `property $\epsilon$' by Sen \cite{1977:sen} who attributes this designation to Douglas Blair.

\begin{property}[Preservation of non-optimality under the addition of elements]\label{prop:preservation:under:addition:of:elements}
For any event $A\neq\emptyset$, and any non-empty finite $A$-consistent sets of gambles $\mathcal{X}$ and $\mathcal{Y}$,
\begin{equation*}
\mathcal{Y} \subseteq \mathcal{X} \Rightarrow \opt(\mathcal{Y}|A) \supseteq \opt(\mathcal{X}|A) \cap \mathcal{Y}.
\end{equation*}
\end{property}

This is called `property $\alpha$' by Sen \cite{1977:sen}, Axiom 7 by Luce and Raiffa \cite[p.~288]{1957:luce}, and `independence of irrelevant alternatives' by Radner and Marschak \cite{1954:radnermarschak}.\footnote{This is different from several other properties bearing the same name, such as that of Arrow's Impossibility Theorem. For further discussion, see Ray \cite{1973:ray}.} It states that any gamble that is non-optimal in a set of gambles $\mathcal{Y}$ is non-optimal in any set of gambles containing $\mathcal{Y}$.

Properties~\ref{prop:insensitivity:to:non:optimal:elements} and~\ref{prop:preservation:under:addition:of:elements} are together equivalent to the well-known property of path independence, which can often be checked more conveniently.

\begin{property}\label{prop:path:independence}
 A choice function $\opt$ is \emph{path independent} if, for any non-empty event $A$, and for any finite family of non-empty finite $A$-consistent sets of gambles $\mathcal{X}_1$, \dots, $\mathcal{X}_n$,
\begin{equation*}
 \opt\Bigg(\bigcup_{i=1}^n\mathcal{X}_i\Bigg|A\Bigg) = \opt\Bigg(\bigcup_{i=1}^n\opt(\mathcal{X}_i|A) \Bigg| A\Bigg).
\end{equation*}
\end{property}

Path independence appears frequently in the social choice literature. Plott \cite{1973:plott} gives a detailed investigation of path independence and its possible justifications. Path independence is also equivalent to Axiom $7'$ of Luce and Raiffa \cite[p.~289]{1957:luce}.

\begin{lemma}[Sen \protect{\cite[Proposition~19]{1977:sen}}]
  \label{lemma:opt:of:unions:equality}
  A choice function $\opt$ satisfies Properties~\ref{prop:insensitivity:to:non:optimal:elements} and~\ref{prop:preservation:under:addition:of:elements} if and only if $\opt$ satisfies Property~\ref{prop:path:independence}.
\end{lemma}

\begin{property}[Backward mixture property]\label{prop:backward:mixture:property}
For any events $A$ and $B$ such that $B\cap A \neq \emptyset$ and $B\cap\compl{A}\neq\emptyset$, any $B\cap\compl{A}$-consistent gamble $Z$, and any non-empty finite $B\cap A$-consistent set of gambles $\mathcal{X}$,
\begin{equation*}
\opt\left(A\mathcal{X} + \compl{A}Z|B\right) \subseteq A\opt(\mathcal{X}|A\cap B) + \compl{A}Z.
\end{equation*}
\end{property}

\begin{theorem}[Backward induction theorem]\label{thm:backopt:normopt:equivalence}
  Let $\opt$ be any choice function.
  The following conditions are equivalent.
  \begin{enumerate}[label=(\Alph*)]
  \item\label{thm:backopt:normopt:equivalence:backoptisnormopt} For any consistent decision tree $\tree$, it holds that $\backopt(\tree)=\normoper_{\opt}(\tree)$.
  \item\label{thm:backopt:normopt:equivalence:properties} $\opt$ satisfies Properties~\ref{prop:backward:conditioning:property}, \ref{prop:insensitivity:to:non:optimal:elements}, \ref{prop:preservation:under:addition:of:elements}, and~\ref{prop:backward:mixture:property}.
  \end{enumerate}
\end{theorem}

\section{Application to Coherent Lower Previsions}
\label{sec:ip:application}

In this section we investigate which of the choice functions for coherent lower previsions satisfy the conditions of Theorem~\ref{thm:backopt:normopt:equivalence}\ref{thm:backopt:normopt:equivalence:properties}. Some of the results are based on proofs for more general choice functions that can be found in the Appendix.

\subsection{Maximality}

Maximality is a strict partial order, so Properties~\ref{prop:insensitivity:to:non:optimal:elements} and~\ref{prop:preservation:under:addition:of:elements} hold, by Proposition~\ref{prop:maximality:path:independence}.

\begin{proposition}\label{prop:maximality:backward:conditioning:property}
  Maximality satisfies Property~\ref{prop:backward:conditioning:property}.
\end{proposition}

\begin{proof}
  We prove a stronger result. Let $A$ be a non-empty event, $\mathcal{X}$ be a non-empty finite set of $A$-consistent gambles, and $\{X,Y\}\subseteq\mathcal{X}$ with $AX=AY$. We show that, for any event $B$ such that $A\cap B \ne \emptyset$, $X\in\optmax(\mathcal{X}|A\cap B)$ implies $Y\in\optmax(\mathcal{X}|A \cap B)$.
  
  If $X\in\optmax(\mathcal{X}|A\cap B)$, then for every $Z\in\mathcal{X}$, $\lpr(Z-X|A \cap B)\leq 0$. But $AX = AY$ implies $(A\cap B) X= (A \cap B) Y$, and, by Proposition~\ref{prop:lpr}\ref{prop:lpr:ax-is-ay}, $\lpr(Z-X|A \cap B)=\lpr(Z-Y|A \cap B)$, and so it immediately follows that $Y\in\optmax(\mathcal{X}|A \cap B)$.
\end{proof}

\begin{proposition}\label{prop:maximality:backward:mixture:property}
 Maximality satisfies Property~\ref{prop:backward:mixture:property}.
\end{proposition}

\begin{proof}
 Consider events $A$ and $B$ such that $A\cap B\neq \emptyset$ and $\compl{A}\cap B\neq\emptyset$, a non-empty finite set of $A\cap B$-consistent gambles $\mathcal{X}$, and an $\compl{A}\cap B$-consistent gamble $Z$. To establish Property~\ref{prop:backward:mixture:property}, it suffices to show that for any $Y\in\mathcal{X}$,
\begin{equation*}
Y \notin \optmax(\mathcal{X}|A\cap B)
\implies
 AY+\compl{A}Z \notin \optmax (A\mathcal{X} \setplus \compl{A}Z | B).
\end{equation*}
If $Y \notin \optmax(\mathcal{X}|A\cap B)$ then there is an $X \in \mathcal{X}$ with $\lpr(X-Y|A\cap B) > 0$. The result follows if we show that $\lpr(AX + \compl{A}Z - (AY + \compl{A}Z) | B) > 0$.
By Proposition~\ref{prop:lpr}\ref{prop:lpr:addition}--\ref{prop:lpr:gbr}, 
\begin{align*}
 0 &=\lpr(A (X-Y - \lpr(X-Y|A\cap B) ) | B )\\
&\leq \lpr(A ( X-Y) | B) + \upr( -A \lpr(X - Y | A\cap B) | B) \\
&= \lpr(A(X-Y)|B) - \lpr(A \lpr(X-Y | A\cap B) | B) \\
&= \lpr(A(X-Y)|B) - \lpr(A|B) \lpr(X-Y|A\cap B)
\end{align*}
where we relied on $\lpr(X-Y|A\cap B)>0$ in the last step. So, indeed
\begin{equation*}
 \lpr(A(X-Y)|B) \geq \lpr(A|B) \lpr(X-Y|A\cap B) > 0.
\end{equation*}
\end{proof}

\begin{corollary}\label{cor:backward:induction:works:for:maximality}
   For any consistent decision tree $\tree$, it holds that $$\backmax(\tree)=\normoper_{\optmax}(\tree).$$
\end{corollary}
\begin{proof}
  Immediate, from Propositions~\ref{prop:maximality:path:independence}, \ref{prop:maximality:backward:conditioning:property}, and~\ref{prop:maximality:backward:mixture:property}, and Theorem~\ref{thm:backopt:normopt:equivalence}.
\end{proof}

\subsection{E-admissibility}

Since E-admissibility is a union of maximality choice functions we have:

\begin{corollary}
 For any consistent decision tree $\tree$, it holds that $$\backE(\tree) = \normoper_{\optE}(\tree).$$
\end{corollary}
\begin{proof}
 Immediate, from Proposition~\ref{prop:opt:union:preserves:properties}, Corollary~\ref{cor:backward:induction:works:for:maximality}, and Theorem~\ref{thm:backopt:normopt:equivalence}.
\end{proof}

Further, from Theorem~\ref{thm:backopt1:subseteq:backopt2}, we have:
\begin{corollary}
\label{cor:e:admissibility:from:maximality:backward:induction}
  For any consistent decision tree $\tree$,
  \begin{equation*}
    \normoper_{\optE}(\tree)=\normoper_{\optE}(\back_{\optmax}(\tree)).
  \end{equation*}
\end{corollary}

\subsection{Interval Dominance}

By Proposition~\ref{prop:maximality:path:independence}, interval dominance satisfies Properties~\ref{prop:insensitivity:to:non:optimal:elements} and~\ref{prop:preservation:under:addition:of:elements}, and it satisfies Property~\ref{prop:backward:conditioning:property} because $AX=AY$ implies $\lpr(X|A)=\lpr(Y|A)$ and $\upr(X|A)=\upr(Y|A)$. We now show that interval dominance fails Property~\ref{prop:backward:mixture:property}.

\begin{example}\label{exam:intdom}
  Suppose $A$ and $B$ are events, and $X$, $Y$, and $Z$ are the gambles given in Table~\ref{tab:interval:dominance:example}. Let $\domlinprevs$ contain all mass functions $P$ such that $A$ and $B$ are independent, $1/4\leq P(A) \leq 3/4$, and $P(B) = 1/2$. Let $\lpr$ be the lower envelope of $\domlinprevs$.

  \begin{table}
    \begin{center}
    \hfill
    \begin{tabular}{r|cc}
      & $A$ & $\compl{A}$ \\
      \hline
      $X$ & $1$ & $1$ \\
      $Y$ & $1.5$ & $3.5$ \\
      $Z$ & $0$ & $4$
    \end{tabular}
    \hfill
    \begin{tabular}{r|cc}
      & $\lpr(\cdot|B)$ & $\upr(\cdot|B)$ \\
      \hline
      $X$ & $1$ & $1$ \\
      $Y$ & $2$ & $3$ \\
      $Z$ & $1$ & $3$ \\
    \end{tabular}
    \hfill
    \begin{tabular}{r|cc}
      & $\lpr$ & $\upr$ \\
      \hline
      $BX+\compl{B}Z$ & $1$ & $2$ \\
      $BY+\compl{B}Z$ & $1.5$ & $3$
    \end{tabular}
    \hfill
    \end{center}
    \caption{Gambles and their lower and upper previsions for Example~\ref{exam:intdom}.}
    \label{tab:interval:dominance:example}
  \end{table}

  Lower and upper previsions of relevant gambles are given in Table~\ref{tab:interval:dominance:example}; for example,
  \begin{align*}
    \upr(BY+\compl{B}Z)
    &=
    \max_{P\in\domlinprevs}
    P(BY+\compl{B}Z)
    =
    \max_{P\in\domlinprevs}
    P(B)P(Y|B)+P(\compl{B})P(Z|\compl{B})
    \\
    &=
    \frac{1}{2}
    \max_{P\in\domlinprevs}(P(Y)+P(Z))
    =
    \frac{1}{2}\max_{p\in[\frac{1}{4},\frac{3}{4}]}(1.5(1-p)+3.5p+4p)
    =
    3
  \end{align*}
  and similar for all other gambles. Clearly, $Y$ interval dominates $X$ conditional on $B$, however, $BY+\compl{B}Z$ does not interval dominate $BX+\compl{B}Z$, violating Property~\ref{prop:backward:mixture:property}.
\end{example}

Even though interval dominance violates Property~\ref{prop:backward:mixture:property}, it can still be of use in backward induction. It is easily shown that (see for instance Troffaes \cite{2007:troffaes})
\begin{equation*}
 \optmax(\mathcal{X}|A) \subseteq \optint(\mathcal{X}|A).
\end{equation*}
By Theorem~\ref{thm:backopt1:subseteq:backopt2}, we therefore have:

\begin{corollary}
  For any consistent decision tree $\tree$,
  \begin{align*}
    \normoper_{\optmax}(\tree)
    &=\normoper_{\optmax}(\back_{\optint}(\tree))
    \\
    \normoper_{\optE}(\tree)
    &=\normoper_{\optE}(\back_{\optint}(\tree))
  \end{align*}
\end{corollary}

It can also be shown that $\back_{\optint}(\tree)\subseteq\normoper_{\optint}(\tree)$ for all $\tree$, so all strategies found by backward induction will be optimal with respect to $\optint$.

\subsection{$\Gamma$-maximin}

$\Gamma$-maximin fails Theorem~\ref{thm:backopt:normopt:equivalence}\ref{thm:backopt:normopt:equivalence:backoptisnormopt}: see for example Seidenfeld \cite[Sequential Example~1, pp.~75--77]{2004:seidenfeld}. Since $\Gamma$-maximin is induced by an ordering, it satisfies Properties~\ref{prop:insensitivity:to:non:optimal:elements} and~\ref{prop:preservation:under:addition:of:elements} by Proposition~\ref{prop:maximality:path:independence}. As for interval dominance, $\Gamma$-maximin satisfies Property~\ref{prop:backward:conditioning:property}. Hence, $\Gamma$-maximin must fail Property~\ref{prop:backward:mixture:property}. Indeed, backward induction can fail in a particularly serious way: it can select a single gamble that is inferior to another normal form gamble. Hence, backward induction may not find any $\Gamma$-maximin gambles.

\section{The Oil Wildcatter Example}
\label{sec:example:two}

We now illustrate our algorithm using the same example as Kikuti et al.~\cite[Fig.~2]{2005:kikuti}. Fig.~\ref{fig:oil:tree} depicts the decision tree, with utiles in units of \$10000. The subject must decide whether to drill for oil ($d_2$) or not ($d_1$). Drilling costs $7$ and provides a return of $0$, $12$, or $27$ depending on the richness of the site. The events $S_1$ to $S_3$ represent the different yields, with $S_1$ being the least profitable and $S_3$ the most. The subject may pay $1$ to test the site before deciding whether to drill; this gives one of three results $T_1$ to $T_3$, where $T_1$ is the most pessimistic and $T_3$ the most optimistic.

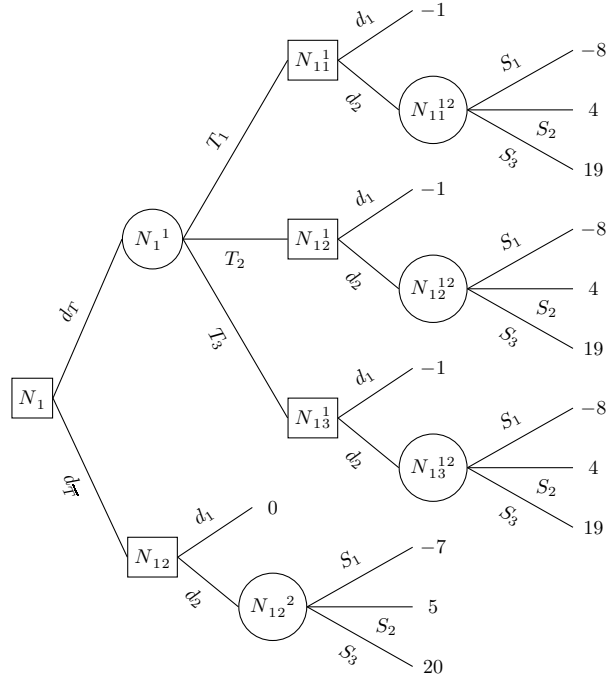
\begin{figure}
  \begin{center}
    \begin{tikzpicture}
      [minimum size=2em,parent anchor=east,child anchor=west,grow'=east,scale=0.75,transform shape]
        \node[draw,rectangle]{$\decnode[1]$}
        [sibling distance=16em,level distance=6em]
        child{
          node[draw,circle]{$\chancenode[1](1)$}
          [sibling distance=9em,level distance=8em]
          child{
            node[draw,rectangle]{$\decnode[1](1)[1]$}
            [sibling distance=5em,level distance=6em]
            child{
              node{$-1$}
              edge from parent
              node[above,sloped]{$d_1$}
            }
            child{
              node[draw,circle]{$\chancenode[1](1)[1](2)$}
              [sibling distance=3em,level distance=8em]
              child{
                node{$-8$}
                edge from parent
                node[above,sloped]{$S_1$}
              }
              child{
                node{$4$}
                edge from parent
                node[below,near end]{$S_2$}
              }
              child{
                node{$19$}
                edge from parent
                node[below,sloped]{$S_3$}
              }
              edge from parent
              node[below,sloped]{$d_2$}
            }
            edge from parent
            node[above,sloped]{$T_1$}
          }
          child{
            node[draw,rectangle]{$\decnode[1](1)[2]$}
            [sibling distance=5em,level distance=6em]
            child{
              node{$-1$}
              edge from parent
              node[above,sloped]{$d_1$}
            }
            child{
              node[draw,circle]{$\chancenode[1](1)[2](2)$}
              [sibling distance=3em,level distance=8em]
              child{
                node{$-8$}
                edge from parent
                node[above,sloped]{$S_1$}
              }
              child{
                node{$4$}
                edge from parent
                node[below,near end]{$S_2$}
              }
              child{
                node{$19$}
                edge from parent
                node[below,sloped]{$S_3$}
              }
              edge from parent
              node[below,sloped]{$d_2$}
            }
            edge from parent
            node[below,sloped]{$T_2$}
          }
          child{
            node[draw,rectangle]{$\decnode[1](1)[3]$}
            [sibling distance=5em,level distance=6em]
            child{
              node{$-1$}
              edge from parent
              node[above,sloped]{$d_1$}
            }
            child{
              node[draw,circle]{$\chancenode[1](1)[3](2)$}
              [sibling distance=3em,level distance=8em]
              child{
                node{$-8$}
                edge from parent
                node[above,sloped]{$S_1$}
              }
              child{
                node{$4$}
                edge from parent
                node[below,near end]{$S_2$}
              }
              child{
                node{$19$}
                edge from parent
                node[below,sloped]{$S_3$}
              }
              edge from parent
              node[below,sloped]{$d_2$}
            }
            edge from parent
            node[below,sloped]{$T_3$}
          }
          edge from parent
          node[above,sloped]{$d_T$}
        }
        child{
          node[draw,rectangle]{$\decnode[1][2]$}
          [sibling distance=5em,level distance=6em]
          child{
            node{$0$}
            edge from parent
            node[above,sloped]{$d_1$}
          }
          child{
            node[draw,circle]{$\chancenode[1][2](2)$}
            [sibling distance=3em,level distance=8em]
              child{
                node{$-7$}
                edge from parent
                node[above,sloped]{$S_1$}
              }
              child{
                node{$5$}
                edge from parent
                node[below,near end]{$S_2$}
              }
              child{
                node{$20$}
                edge from parent
                node[below,sloped]{$S_3$}
              }
            edge from parent
            node[below,sloped]{$d_2$}
          }
          edge from parent
          node[below,sloped]{$d_{\compl{T}}$}
        };
    \end{tikzpicture}
    \caption{Decision tree for the oil wildcatter.}
  \label{fig:oil:tree}
  \end{center}
\end{figure}

Lower and upper probabilities are given for each $T_i$ (Table~\ref{tab:oil:t}), and for each $S_i$ conditional on $T_i$ (Table~\ref{tab:oil}). (Some intervals are tighter than those in Kikuti et al., since their values are incoherent---we corrected these by natural extension \cite[\S 3.1]{1991:walley}.)

\begin{table}
  \begin{center}
    \begin{tabular}{ccc}
      $T_1$ & $T_2$ & $T_3$ \\
      \hline
      $0.183$, $0.222$ &$0.333$, $0.363$ &$0.444$, $0.454$
    \end{tabular}
  \caption{Unconditional lower and upper probabilities $\lpr(T_i)$ and $\upr(T_i)$ for oil example.}
  \label{tab:oil:t}
  \end{center}
\end{table}

\begin{table}
  \begin{center}
    \begin{tabular}{r|ccc}
      & $T_1$ & $T_2$ & $T_3$ \\
      \hline
      $S_1$ &$0.547$, $0.653$& $0.222$, $0.333$& $0.111$, $0.166$ \\
      $S_2$ &$0.222$, $0.272$ & $0.363$, $0.444$& $0.333$, $0.363$ \\
      $S_3$ &$0.125$, $0.181$ &$0.250$, $0.363$ &$0.471$, $0.556$
    \end{tabular}
  \caption{Conditional lower and upper probabilities $\lpr(S_i|T_i)$ and $\upr(S_i|T_i)$ for oil example.}
  \label{tab:oil}
  \end{center}
\end{table}

By \emph{marginal extension}~\cite[\S 6.7.2]{1991:walley}, the lower prevision of a gamble $Z$ is then
\begin{equation*}
  \lpr(Z)=\lpr(T_1\lpr(Z|T_1) + T_2\lpr(Z|T_2) + T_3\lpr(Z|T_3)).
\end{equation*}

Let $X=-7 S_1 + 5 S_2 + 20 S_3$, and again let $\tree^*_*=\subtreeat{\tree}{\decnode^*_*}$. Since we will only be concerned with maximality, and normal form decisions in this problem are uniquely identified by their gambles, we can conveniently work with gambles in this example. Therefore, we use the following notation:
\begin{align*}
  \opt&=\optmax & \back&=\normgambles\circ\backmax & \normoper&=\normgambles\circ\normoper_{\optmax}
\end{align*}

\begin{figure}
  \begin{center}
    \begin{enumerate}[label=(\roman*)]
    \item
    \begin{tikzpicture}
      [baseline=0,minimum size=2em,parent anchor=east,child anchor=west,grow'=east,scale=0.75,transform shape]
        \node[draw,rectangle]{$\decnode[1]$}
        [sibling distance=7em,level distance=6em]
        child{
          node[draw,circle]{$\chancenode[1](1)$}
          [sibling distance=4em,level distance=8em]
          child{
            node[draw,rectangle]{$\decnode[1](1)[1]$}
            [sibling distance=2em,level distance=6em]
            child{
              node[right]{$\{-1\}$}
              edge from parent
              node[above,sloped]{$d_1$}
            }
            child{
              node[right]{$\{X-1\}$}
              edge from parent
              node[below,sloped]{$d_2$}
            }
            edge from parent
            node[above,sloped]{$T_1$}
          }
          child{
            node[draw,rectangle]{$\decnode[1](1)[2]$}
            [sibling distance=2em,level distance=6em]
            child{
              node[right]{$\{-1\}$}
              edge from parent
              node[above,sloped]{$d_1$}
            }
            child{
              node[right]{$\{X-1\}$}
              edge from parent
              node[below,sloped]{$d_2$}
            }
            edge from parent
            node[below,sloped]{$T_2$}
          }
          child{
            node[draw,rectangle]{$\decnode[1](1)[3]$}
            [sibling distance=2em,level distance=6em]
            child{
              node[right]{$\{-1\}$}
              edge from parent
              node[above,sloped]{$d_1$}
            }
            child{
              node[right]{$\{X-1\}$}
              edge from parent
              node[below,sloped]{$d_2$}
            }
            edge from parent
            node[below,sloped]{$T_3$}
          }
          edge from parent
          node[above,sloped]{$d_T$}
        }
        child{
          node[draw,rectangle]{$\decnode[1][2]$}
          [sibling distance=2em,level distance=6em]
          child{
            node[right]{$\{0\}$}
            edge from parent
            node[above,sloped]{$d_1$}
          }
          child{
            node[right]{$\{X\}$}
            edge from parent
            node[below,sloped]{$d_2$}
          }
          edge from parent
          node[below,sloped]{$d_{\compl{T}}$}
        };
    \end{tikzpicture}
    \item
    \begin{tikzpicture}
      [baseline=0,minimum size=2em,parent anchor=east,child anchor=west,grow'=east,scale=0.75,transform shape]
        \node[draw,rectangle]{$\decnode[1]$}
        [sibling distance=5em,level distance=6em]
        child{
          node[draw,circle]{$\chancenode[1](1)$}
          [sibling distance=2em,level distance=8em]
          child{
            node[right]{$\opt(\{-1,X-1\}|T_1)=\{-1,X-1\}$}
            edge from parent
            node[above,sloped,pos=0.9]{$T_1$}
          }
          child{
            node[right]{$\opt(\{-1,X-1\}|T_2)=\{X-1\}$}
            edge from parent
            node[below,sloped,pos=0.9]{$T_2$}
          }
          child{
            node[right]{$\opt(\{-1,X-1\}|T_3)=\{X-1\}$}
            edge from parent
            node[below,sloped,pos=0.9]{$T_3$}
          }
          edge from parent
          node[above,sloped]{$d_T$}
        }
        child{
          node[right]{$\opt(\{0,X\})=\{X\}$}
          edge from parent
          node[below,sloped]{$d_{\compl{T}}$}
        };
    \end{tikzpicture}
    \item
    \begin{tikzpicture}
      [baseline=0,minimum size=2em,parent anchor=east,child anchor=west,grow'=east,scale=0.75,transform shape]
        \node[draw,rectangle]{$\decnode[1]$}
        [sibling distance=3em,level distance=6em]
        child{
          node[right]{$\opt(\{T_1(-1)+T_2(X-1)+T_3(X-1),T_1(X-1)+T_2(X-1)+T_3(X-1)\})=\{(T_2+T_3)X-1,X-1\}$}
          edge from parent
          node[above,sloped]{$d_T$}
        }
        child{
          node[right]{$\{X\}$}
          edge from parent
          node[below,sloped]{$d_{\compl{T}}$}
        };
    \end{tikzpicture}
    \item
    \begin{tikzpicture}
      [baseline=0,minimum size=2em,parent anchor=east,child anchor=west,grow'=east,scale=0.75,transform shape]
        \node{$\opt(\{(T_2+T_3)X-1,X-1,X\})=\{X\}$};
    \end{tikzpicture}
      \end{enumerate}
    \caption{Solving the oil wildcatter example by normal form backward induction.}
  \label{fig:oil:tree:solution}
  \end{center}
\end{figure}
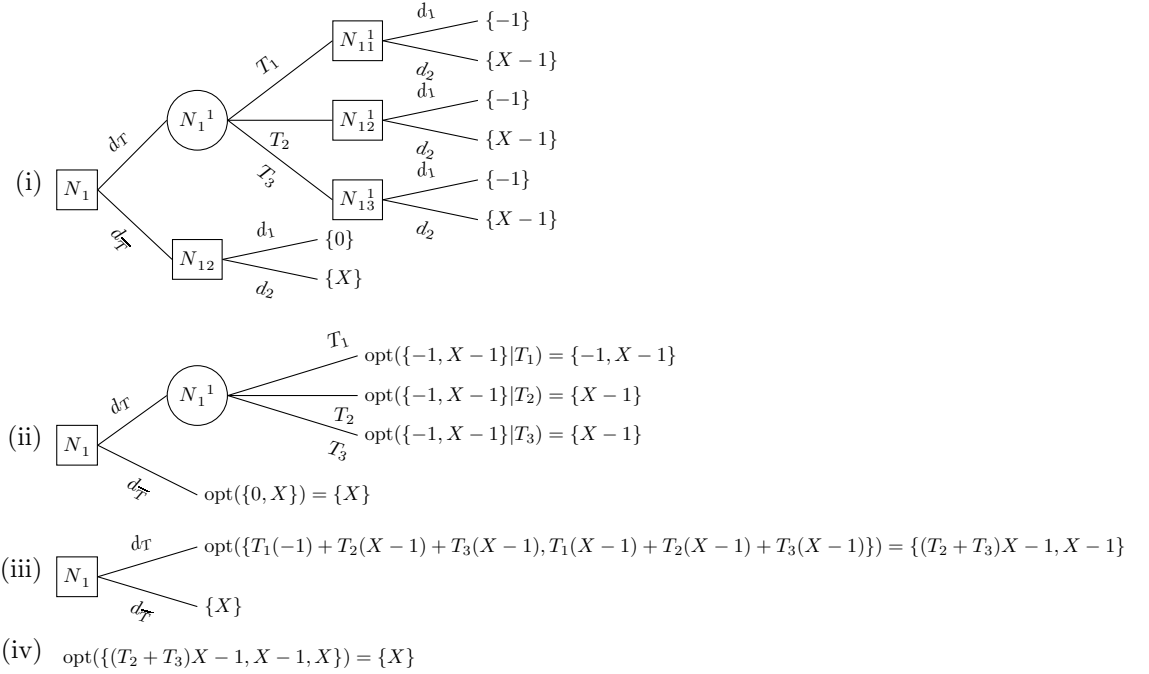

Fig.~\ref{fig:oil:tree:solution} depicts the process of backward induction described next.
\begin{enumerate}[label=(\roman*)]
\item
Of course, $\back(\cdot)$ at the final chance nodes simply reports the gamble: $\back(\subtree[1](1)[1](2))=\back(\subtree[1](2)[1](2))=\back(\subtree[1](3)[1](2))=\{X-1\}$, and $\back(\subtree[1][2](2))=\{X\}$.
\item 
For $\subtree[1](1)[1]$, we must find $\lpr((X-1) - (-1)|T_1)$ and $\lpr(-1 - (X-1) |T_1)$. These lower previsions can be computed using Table~\ref{tab:oil} as follows: $X$ will have lowest expected value when the worst outcome $S_1$ is most likely (probability $0.653$) and the best outcome $S_3$ is least likely (probability $0.125$), and so the probability of $S_2$ is $0.222$. So, $\lpr(X|T_1)=-7\times0.653 + 5\times0.222 + 20\times0.125=-0.961$. Similarly, $\lpr(-X|T_1)=-1.151$. Neither of these is positive, so $\back(\subtree[1](1)[1])=\{X-1, -1\}$.

For $\subtree[1](1)[2]$, $\lpr((X-1)-(-1)|T_2)=4.754$, and therefore $d_2$ dominates $d_1$, so $\back(\subtree[1](1)[2])=\{X-1\}$. Similarly, $\lpr((X-1)-(-1)|T_3)=10.073$, so $\back(\subtree[1](1)[3])=\{X-1\}$.

For $\subtree[1][2]$, we need to find $\lpr(X-0)$. By marginal extension we have
\begin{align*}
  \lpr(X)&=\lpr(T_1 \lpr(X|T_1)+ T_2 \lpr(X|T_2)+ T_3\lpr(X|T_3)) \\
  &= \lpr(-0.961 T_1 + 4.754 T_2 + 10.073 T_3) \\
  &= 0.222 \times -0.961 + 0.334 \times 4.754 + 0.444 \times 10.073
  = 5.846906.
\end{align*}
This is greater than zero, so $\back(\subtree[1][2])=\{X\}$.

\item
At $\subtree[1](1)$ there are two potentially optimal gambles: $T_1 (X-1)+ T_2 (X-1) + T_3 (X-1)=X-1$ and $T_1(-1) + T_2 (X-1) + T_3(X-1)=(T_2+T_3)X-1$. We must find $\lpr((X-1)-((T_2+T_3)X-1))=\lpr(T_1 X)$ and $\lpr(((T_2+T_3)X-1)-(X-1))=\lpr(-T_1X)$. Using marginal extension,
\begin{align*}
  \lpr(T_1 X)&=\lpr(T_1 \lpr(X|T_1))=\lpr(-0.961 T_1)=0.222\times -0.961=-0.213342<0, \\
  \lpr(-X T_1)&=\lpr(T_1\lpr(-X|T_1))=\lpr(-1.151 T_1)=0.222\times -1.151=-0.255522<0,
\end{align*}
so $\back(\subtree[1](1))=\{X-1, (T_2+T_3)X-1\}$.

\item
Finally, for $\tree$, we must consider $\{X, X-1, (T_2+T_3)X-1\}$. It is clear that $\lpr(X- (X-1))=1>0$, so $X-1$ can be eliminated. It is also clear that if a gamble does not dominate $X-1$ then is also does not dominate $X$, so by our calculation at $\subtree[1](1)$ we know that $X$ is maximal. We finally have
\begin{align*}
  \lpr(X-((T_2+T_3)X-1)) &= \lpr(T_1 X + 1)=\lpr(T_1X)+1=
 -0.213342+1>0,
\end{align*}
so $\back(\tree)=\{X\}$. So, the optimal strategy is: do not test and just drill.
\end{enumerate}

We found a single maximal strategy. By Corollary~\ref{cor:e:admissibility:from:maximality:backward:induction}, it is also the unique E-admissible strategy. (Our solution differs from Kikuti et al. \cite{2005:kikuti}; since they do not detail their calculations, we could not identify why.) Of course, if the imprecision was larger, we would have found more, but it does show that non-trivial sequential problems can give unique solutions even when probabilities are imprecise.

In this example, the usual normal form method requires comparing 10 gambles at once. By normal form backward induction, we only had to compare 2 gambles at once at each stage (except at the end, where we had 3), leading us much quicker to the solution: the computational benefit of normal form backward induction is obvious.

\section{Conclusion}
\label{sec:conclusion}

When solving sequential decision problems with limited knowledge, it may be impossible to assign a probability to every event. An alternative is to use a coherent lower prevision, or, equivalently, a closed convex set of probability mass functions. Under such a model there are several plausible generalizations of maximizing expected utility.

Given any criterion, we considered two methods of solving decision trees: the usual normal form, where the subject applies the criterion to the lists all strategies, and normal form backward induction, adapted from \cite{2005:kikuti}. If they coincide, backward induction helps efficiently solving trees. If they differ, doubt is cast on the criterion's suitability.

In Theorem~\ref{thm:backopt:normopt:equivalence} we identified when the two methods coincide. We then applied these results to the choice functions for coherent lower previsions. As was already known, $\Gamma$-maximin fails Property~\ref{prop:backward:mixture:property}. Interval dominance fails the same condition. However, and perhaps surprisingly, maximality and E-admissibility satisfy all conditions, in the case where lower probabilities are non-zero. If any lower probabilities are zero, then Property~\ref{prop:backward:mixture:property} fails (unsurprisingly, as in this case it already fails with precise probabilities).

When analysing choice functions, whether for lower previsions or not, usually Property~\ref{prop:backward:mixture:property} is most troublesome. Failing Property~\ref{prop:backward:conditioning:property} would involve an unnatural form of conditioning, and path independence (Property~\ref{prop:path:independence}) is a very natural consistency condition that one usually wants to satisfy before even considering decision trees.

We have not argued that a normal form operator, and $\normoper_{\opt}$ in particular, gives the best solution to a decision tree. A normal form solution requires a policy to be specified and adhered to. The subject does this policy only by her own resolution: she of course can change her policy upon reaching a decision node \cite{1988:seidenfeld}. One might argue that a normal form solution is only acceptable when the subject cannot change her mind (for example, if she instructs, in advance, others to carry out the actions).

Further, many choice functions cause $\normoper_{\opt}$ to have undesirable properties, even when they satisfy Properties~\ref{prop:backward:conditioning:property}, \ref{prop:insensitivity:to:non:optimal:elements}, \ref{prop:preservation:under:addition:of:elements}, and~\ref{prop:backward:mixture:property}. For example, using maximality or $E$-admissibility with a lower prevision allows the subject to choose to pay for apparently irrelevant information instead of making an immediate decision \cite{2007:seidenfeld}, and a gamble that is optimal in a subtree may become non-optimal in the full tree \cite{1988:hammond,2009:huntley:troffaes::isipta,2011:huntley:subtree:perfectness}.

Moreover, normal form backward induction does not always help with computations. The need to store, at every stage, all optimal gambles, could be a burden. Secondly, if imprecision is large, causing only few gambles to be deleted, the set of optimal gambles at each stage will still eventually become too large. In such situations, a form of approximation may be necessary. Even so, we have shown that, perhaps surprisingly, the normal form \emph{can} be solved exactly with backward induction, and when either trees or imprecision are not too large, the method will be computationally feasible.

\bibliography{imprecisetrees} 
\bibliographystyle{amsplainurl}

\newpage

\appendix
\numberwithin{lemma}{section}

\section{Results for General Choice Functions}
\label{sec:appendix}

This appendix details intermediate results required for the proofs in Section~\ref{sec:ip:optimality}. Since the results are applicable for choice functions that are nothing to do with coherent lower previsions, and so may be useful for investigating other uncertainty models, we present them separately.

\begin{proposition}\label{prop:maximality:path:independence}
  For each non-empty event $A$, let $\succ_A$ be any strict partial order on $A$-consistent gambles. The choice function induced by these strict partial orders, that is,
  \begin{equation*}
    \optsucc(\mathcal{X}|A)
    =
    \{X\in\mathcal{X}\colon (\forall Y\in\mathcal{X})(Y\not\succ_A X)\}
  \end{equation*}
  satisfies Properties~\ref{prop:insensitivity:to:non:optimal:elements} and~\ref{prop:preservation:under:addition:of:elements}.
\end{proposition}

\begin{proof}
By Lemma~\ref{lemma:opt:of:unions:equality}, it suffices to show that $\optsucc$ is path independent. Let $\mathcal{X}_1$, \dots, $\mathcal{X}_n$ be non-empty finite sets of $A$-consistent gambles, and let $A$ be a non-empty event. Let $\mathcal{X}=\bigcup_{i=1}^n \mathcal{X}_i$ and $\mathcal{Z}=\bigcup_{i=1}^n \optsucc(\mathcal{X}_i | A)$. We show must show that
\begin{equation}
  \label{eq:prop:maximality:path:independence:helper:1}
  \optsucc(\mathcal{X}|A)=\optsucc(\mathcal{Z}|A).
\end{equation}

By definition,
\begin{align*}
  \optsucc(\mathcal{Z}|A) &= \{Z \in\mathcal{Z}\colon(\forall Y\in\mathcal{Z})(Y\not\succ_A Z)\},\\
\intertext{and, observe that, if $X\in\mathcal{X}$ but $X\notin\mathcal{Z}$, by transitivity of $\succ_A$ and finiteness of $\mathcal{X}$, there is a $Y\in\mathcal{Z}$ such that $Y\succ_A X$. Therefore again by transitivity of $\succ_A$, for any $Z\in\mathcal{Z}$ such that $X \succ_A Z$, we have $Y\succ_A Z$. So,}
  &= \{Z \in\mathcal{Z}\colon(\forall Y\in\mathcal{X})(Y\not\succ_A Z)\},\\
\intertext{and once again, by definition of $\mathcal{Z}$, if $X\in\mathcal{X}$ but $X\notin\mathcal{Z}$ there is a $Y\in\mathcal{X}$ such that $Y\succ_A X$, so we have}
  &= \{X \in \mathcal{X}\colon(\forall Y\in\mathcal{X})(Y\not\succ_A X)\} \\
  &= \optsucc(\mathcal{X}|A).
\end{align*}

\end{proof}

\begin{proposition}\label{prop:opt:union:preserves:properties}
  Let $\{\opt_i\colon i\in\mathcal{I}\}$ be a family of choice functions. For any non-empty event $A$ and any non-empty finite set of $A$-consistent gambles $\mathcal{X}$, let
  \begin{equation*}
    \opt(\mathcal{X}|A) = \bigcup_{i \in \mathcal{I}} \opt_i(\mathcal{X}|A).
  \end{equation*}
  \begin{enumerate}[label=(\roman*)]
  \item\label{prop:opt:union:preserves:properties:insensitivity:to:non:optimal:elements}
  If each $\opt_i$ satisfies Property~\ref{prop:insensitivity:to:non:optimal:elements}, then so does $\opt$.
  \item\label{prop:opt:union:preserves:properties:preservation:under:addition:of:elements}
  If each $\opt_i$ satisfies Property~\ref{prop:preservation:under:addition:of:elements}, then so does $\opt$.
  \item\label{prop:opt:union:preserves:properties:backward:mixture:property}
  If each $\opt_i$ satisfies Property~\ref{prop:backward:mixture:property}, then so does $\opt$.
  \item \label{prop:opt:union:preserves:properties:backward:conditioning:property}
  If each $\opt_i$ satisfies Properties~\ref{prop:backward:conditioning:property}, \ref{prop:preservation:under:addition:of:elements}, and~\ref{prop:backward:mixture:property}, then $\opt$ satisfies Property~\ref{prop:backward:conditioning:property}.
  \end{enumerate}
\end{proposition}
\begin{proof}
  \ref{prop:opt:union:preserves:properties:insensitivity:to:non:optimal:elements}.
  By definition of $\opt$ and by assumption, for any finite non-empty sets of gambles $\mathcal{X}$ and $\mathcal{Y}$ such that $\mathcal{Y}\subseteq\mathcal{X}$ and for any $i \in \mathcal{I}$, $\opt_i(\mathcal{X}|A)\subseteq\opt(\mathcal{X}|A)\subseteq\mathcal{Y}$, and therefore by Property~\ref{prop:insensitivity:to:non:optimal:elements}, $\opt_i(\mathcal{X}|A)=\opt_i(\mathcal{Y}|A)$. Whence,
 \begin{equation*}
  \opt(\mathcal{Y}|A) = \bigcup_{i \in \mathcal{I}} \opt_i(\mathcal{Y}|A) = \bigcup_{i \in \mathcal{I}} \opt_i(\mathcal{X}|A) = \opt(\mathcal{X}|A).
 \end{equation*}
  
  \ref{prop:opt:union:preserves:properties:preservation:under:addition:of:elements}.
  By assumption, for any finite non-empty sets of gambles $\mathcal{X}$ and $\mathcal{Y}$ such that $\mathcal{Y} \subseteq \mathcal{X}$ and for any $i \in \mathcal{I}$, $\opt_i(\mathcal{Y}|A) \supseteq \opt_i(\mathcal{X}|A) \cap \mathcal{Y}$. Therefore,
  \begin{align*}
    \opt(\mathcal{Y}|A) &= \bigcup_{i \in \mathcal{I}} \opt_i(\mathcal{Y}|A) \supseteq \bigcup_{i \in \mathcal{I}} (\opt_i(\mathcal{X}|A)\cap\mathcal{Y}) \\
    &= \mathcal{Y} \cap \bigcup_{i \in \mathcal{I}} \opt_i(\mathcal{X}|A) = \opt(\mathcal{X}|A)\cap\mathcal{Y}.
  \end{align*}
  
  \ref{prop:opt:union:preserves:properties:backward:mixture:property}. By assumption, for any non-empty finite set of gambles $\mathcal{X}$, any gamble $Z$, any events $A$ and $B$ such that $A \cap B \neq \emptyset$, and for any $i \in \mathcal{I}$,
 \begin{equation*}
  \opt_i(A\mathcal{X} \setplus \compl{A}Z|B) \subseteq A\opt_i(\mathcal{X}|A\cap B) \setplus \compl{A} Z,
 \end{equation*}
 whence
 \begin{align*}
  \opt(A\mathcal{X}\setplus\compl{A}Z|B) &= \bigcup_{i \in \mathcal{I}}\opt_i(A\mathcal{X}\setplus\compl{A}Z|B) \\
  &\subseteq \bigcup_{i \in \mathcal{I}}(A\opt_i(\mathcal{X}|A\cap B) \setplus \compl{A} Z) \\
  &= \compl{A}Z\setplus A \bigcup_{i \in \mathcal{I}}\opt_i(\mathcal{X}|A\cap B) \\
  &= \compl{A}Z \setplus A\opt(\mathcal{X}|B).
 \end{align*}
 
 \ref{prop:opt:union:preserves:properties:backward:conditioning:property}. Let $A$ and $B$ be events such that $A\cap B\neq\emptyset$ and $\compl{A}\cap B\neq\emptyset$, $\mathcal{Z}$ be a non-empty finite set of $\compl{A}\cap B$-consistent gambles, and $\mathcal{X}$ be a non-empty finite set of $A\cap B$-consistent gambles such that there is $\{X,Y\}\subseteq \mathcal{X}$ with $AX=AY$. Suppose that there is a $Z\in\mathcal{Z}$ such that $AX+\compl{A}Z\in\opt(A\mathcal{X} + \compl{A}\mathcal{Z} | B)$.
 
 By definition of $\opt$, there is a $j$ such that $AX+\compl{A}Z\in\opt_j(A\mathcal{X} + \compl{A}\mathcal{Z} | B)$. We show that both $X$ and $Y$ are in $\opt_j(\mathcal{X}|A\cap B)$, and therefore are both in $\opt(\mathcal{X}|A\cap B)$. It follows from Properties~\ref{prop:preservation:under:addition:of:elements} and~\ref{prop:backward:mixture:property} that
 \begin{equation*}
   \opt_j(A\mathcal{X} + \compl{A}\mathcal{Z} | B) \subseteq A\opt_j(\mathcal{X}|A\cap B) +\compl{A}\opt_j(\mathcal{Z}|\compl{A}\cap B).
 \end{equation*}
 Therefore, there is a $V\in\opt_j(\mathcal{X}|A\cap B)$ with $AV=AX$. Finally, $\opt_j$ satisfies Property~\ref{prop:backward:conditioning:property}, and therefore both $X$ and $Y$ must be in $\opt_j(\mathcal{X}|A\cap B)$. This establishes Property~\ref{prop:backward:conditioning:property} for $\opt$.

\end{proof}

The final result is the following: if $\opt_1$ satisfies the necessary properties, $\opt_2$ does not, but $\opt_1 \subseteq \opt_2$, then we can use $\opt_1(\back_{\opt_2}(\cdot))$ to find $\normoper_{\opt_1}$. This could be of interest in situations where $\opt_2$ is much more computationally efficient than $\opt_1$, and still eliminates enough gambles to be useful.

\begin{theorem}\label{thm:backopt1:subseteq:backopt2}
 Let $\opt_1$ and $\opt_2$ be choice functions such that $\opt_1$ satisfies Properties~\ref{prop:backward:conditioning:property}, \ref{prop:insensitivity:to:non:optimal:elements}, \ref{prop:preservation:under:addition:of:elements}, and~\ref{prop:backward:mixture:property}, and for any non-empty event $A$ and any non-empty finite set of $A$-consistent gambles $\mathcal{X}$,
 \begin{equation*}
  \opt_1(\mathcal{X}|A)\subseteq\opt_2(\mathcal{X}|A).
 \end{equation*}
 Then, for any consistent decision tree $\tree$,
\begin{equation}
\label{eq:backopt1:subseteq:backopt2:subset}
 \normoper_{\opt_1}(\tree) = \normoper_{\opt_1}(\back_{\opt_2}(\tree)).
\end{equation}
\end{theorem}

\end{document}